\documentclass[11pt]{amsart}
\usepackage{amsmath}
\usepackage{amssymb}
\usepackage{amsthm}
\usepackage{amscd}
\usepackage{amsxtra}
\usepackage{verbatim}
\usepackage{xcolor}
\usepackage{color}
\usepackage{enumerate}
\usepackage{mathrsfs}

\usepackage[colorlinks=true, bookmarks=true,pdfstartview=FitV, linkcolor=blue, citecolor=blue, urlcolor=blue]{hyperref}

\numberwithin{equation}{section}

\RequirePackage{geometry}
\newtheorem{theorem}{Theorem}[section]
\newtheorem{lemma}[theorem]{Lemma}

\newtheorem{proposition}[theorem]{Proposition}

\theoremstyle{definition}

\newtheorem{remark}[theorem]{Remark}

\theoremstyle{remark}

\newcommand{\floor}[1]{\left\lfloor #1 \right\rfloor}
\DeclareMathOperator*{\essinf}{ess\,inf}
\DeclareMathOperator*{\esssup}{ess\,sup}
\DeclareMathOperator{\supp}{\mathrm{supp}}
\newcommand{\Sp}{\mathcal S}

\newcommand{\pp}{{p(\cdot)}}

\newcommand{\Lp}{L^{p(\cdot)}}
\newcommand{\Pp}{\mathcal P}
\newcommand{\qq}{{q(\cdot)}}

\newcommand{\R}{\mathbb R}
\newcommand{\subRn}{{{\mathbb R}^n}}
\newcommand{\B}{{\mathcal B}}
\newcommand{\F}{{\mathcal F}}

\def\Xint#1{\mathchoice 
{\XXint\displaystyle\textstyle{#1}}% 
{\XXint\textstyle\scriptstyle{#1}}% 
{\XXint\scriptstyle\scriptscriptstyle{#1}}% 
{\XXint\scriptscriptstyle\scriptscriptstyle{#1}}% 
\!\int} 
\def\XXint#1#2#3{{\setbox0=\hbox{$#1{#2#3}{\int}$} 
\vcenter{\hbox{$#2#3$}}\kern-.5\wd0}}

\def\avgint{\Xint-}

\allowdisplaybreaks

\begin{document}

\title[Multilinear Calder\'on-Zygmund operators on weighted and Variable
Hardy spaces]{The boundedness of multilinear Calder\'on-Zygmund
  operators on weighted and variable Hardy spaces}

\author[Cruz-Uribe]{David Cruz-Uribe, OFS}
\address{Department of Mathematics, University of Alabama, Tuscaloosa, AL 35487}
\email{dcruzuribe@ua.edu}

\author[Moen]{Kabe Moen}
\address{Department of Mathematics, University of Alabama, Tuscaloosa, AL 35487}
\email{kabe.moen@ua.edu}

\author[Nguyen]{Hanh Van Nguyen}
\address{Department of Mathematics, University of Alabama, Tuscaloosa, AL 35487}
%\address{Department of Mathematics, Hue College of Education, Hue, Vietnam}
\email{hvnguyen@ua.edu}



\thanks{The first author is supported by NSF Grant DMS-1362425 and research funds from the
  Dean of the College of Arts \& Sciences, the University of Alabama. The second author is supported by the Simons Foundation.}

\date{July 20, 2017}

\begin{abstract}
We establish the boundedness of the multilinear Calderon-Zygmund
operators from a product of weighted Hardy spaces into a weighted
Hardy or Lebesgue space.  Our results generalize to the weighted
setting results obtained by Grafakos and Kalton~\cite{GK01} and recent
work by the third author, Grafakos, Nakamura, and
Sawano~\cite{GNNS17A}.   As part of our proof we provide a finite atomic
decomposition theorem for weighted Hardy spaces, which is interesting
in its own right.  As a consequence of our weighted results, we prove
the corresponding estimates on variable Hardy spaces.  Our main tool
is a multilinear extrapolation theorem that generalizes a result of
the first author and Naibo~\cite{CN16}.
\end{abstract}

\maketitle

\section{Introduction}
In this paper we study the boundedness of multilinear Calder\'on-Zygmund operators ($m$-CZOs) on
products of weighted and variable Hardy spaces.   More precisely, we
are interested in the following operators.  Let
$K(y_0,y_1,\ldots,y_m)$  be a kernel that  is defined away from the diagonal $y_0=y_1=\cdots =y_m$ in $(\mathbb{R}^{n})^{m+1}$ and satisfies the smoothness condition
 \begin{equation}\label{eqn:kernelregularity}
 \big| \partial^{\alpha_0}_{y_0}\cdots  \partial^{\alpha_m}_{y_m}
 K(y_0,y_1,\ldots,y_m)\big|
 \le A_{\alpha_0,\ldots,\alpha_m}\Big(
 \sum_{k,l=0}^{m}|y_k-y_l|
 \Big)^{-(mn+|\alpha_0|+\cdots+|\alpha_m|)}
 \end{equation}
 for all $\alpha=(\alpha_0,\ldots,\alpha_m)$ such that
 $|\alpha|=|\alpha_0|+\cdots+|\alpha_m|\leq N$, where $N$ is a
 sufficiently large integer.  An $m$-CZO is a multilinear operator $T$ that
 satisfies
 $$T:L^{q_1}(\mathbb R^n)\times\cdots\times L^{q_m}(\mathbb R^n)\rightarrow L^q(\mathbb R^n)$$
for some $1<q_1,\ldots,q_m<\infty$ and 
$\frac1q=\frac1{q_1}+\cdots+\frac{1}{q_m}$,  and $T$ has the integral representation  
$$T(f_1,\ldots,f_m)(x)=\int_{\mathbb (\mathbb R^n)^m}K(x,y_1,\ldots,y_m)f(y_1)\cdots f(y_m)\,dy_1\cdots dy_m$$
whenever $f_i\in L^\infty_c(\mathbb R^n)$ and $x\notin \cap_i \supp(f_i)$.

Multilinear CZOs were introduced by Coifman and Meyer
\cite{CM75,CM78} in the 1970s and were systematically studied by
Grafakos and Torres \cite{GT02}.  They showed that $m$-CZOs are bounded
from
$L^{p_1}(\mathbb R^n)\times\cdots\times L^{p_m}(\mathbb
R^n)\rightarrow L^p(\mathbb R^n)$, for any $1<p_1,\ldots,p_m<\infty$
and $p$ defined by $\frac1p=\frac{1}{p_1}+\cdots+\frac{1}{p_m}$.
Further, $m$-CZOs satisfy weak endpoint bounds when $p_i=1$ for some $i$.
For Lebesgue space bounds, it is sufficient to take $N=1$
in~\eqref{eqn:kernelregularity} and in fact weaker regularity
conditions are sufficient.
Bounds for $m$-CZOs from products of Hardy spaces into Lebesgue spaces
were proved by Kalton and Grafakos~\cite{GK01} (see also Grafakos and
He~\cite{GH14}).  As in the linear case, more regularity is required on
the operators:  in this case, $N\geq s= {\lfloor
n(\frac{1}{p}-1)\rfloor}_+$ where $x_+=\max(0,x)$.  Very recently, bounds into Hardy spaces were proved
by the third author, Grafakos, Nakamura and Sawano~\cite{GNNS17A}.
To map into Hardy spaces the kernel $K$ must
satisfy~\eqref{eqn:kernelregularity} for 
\[ N > s + \max\bigg\{ \bigg\lfloor mn\Big(\frac{1}{p_k}-1\Big) \bigg\rfloor_+ : 1\leq k \leq m
  \bigg\}+mn. \]
Moreover, in the multilinear case the operator $T$ must satisfy an
additional cancelation condition:
\begin{equation}
\label{eq.can}
\int x^{\alpha}T(a_1,\ldots,a_m)(x)\, dx=0,
\end{equation}
for $|\alpha|\leq s$ and all $(p_k,\infty,N)$ atoms $a_k$.   For linear
CZOs of convolution type, this condition holds automatically:
see~\cite[Lemma~2.1]{GNNS17B}.  An example of a bilinear CZO that
satisfies this cancelation condition is $T=R_1+R_2$, where $R_i$
is the bilinear Riesz transform
\[ R_i(f,g)(x) = \mathrm{pv} \int_{\R}\int_{\R} 
\frac{x-y_i}{|(x-y_1,x-y_2)|^3} 
f(y_1)g(y_2)\,dy_1dy_2. \]
Somewhat surprisingly, neither Riesz transform itself has sufficient
cancellation.  For more examples of convolution-type multilinear
operators that do and do not satisfy this cancelation condition,
see~\cite{GNNS17A,GNNS17B}.\footnote{We note in passing that the
  results for $m$-CZOs in~\cite{GNNS17A} are stated for convolution
  type operators, but as the authors note (see Remark~3.4), their results
  extend to non-convolution type $m$-CZOs.}

\medskip

Weighted norm inequalities for multilinear operators were first
considered by Grafakos and Torres~\cite{GrT02}.  Later, Lerner,
{\em et al.}~\cite{LOPTT09}
characterized the weighted inequalities for $m$-CZOs using a
multilinear generalization of the Muckenhoupt $A_p$ condition.
Weighted Hardy spaces were introduced by Garc\'ia-Cuerva~\cite{GC79}.
A complete treatment of weighted Hardy spaces is due to Str\"omberg
and Torchinsky~\cite{ST89}; they proved that (linear)
Calder\'on-Zygmund operators whose kernels have enough regularity map
$H^p(w)$ into $L^p(w)$ or $H^p(w)$, for $0<p<\infty$ and for weights
$w\in A_\infty$. 

Our goal is to generalize the results of Str\"omberg and
Torchinsky to $m$-CZOs.  To state them, we first define some notation.
To do so we rely on some (hopefully) well-known concepts; complete
definitions will be given below.  
%
% It is well known (see \cite{GT02}) A weight is a non-negative
% a.e. function that is locally integrable.  Given $p>1$ we say a
% weight belongs to the Muckenhoupt class $A_p$ if there is a positive
% finite constant $C$ such that
% $$\left(\frac1{|Q|}\int_Q w(x) \,dx\right)\left(\frac{1}{|Q|}\int_Q w(x)^{1-p'}\,dx\right)^{p-1}\leq C$$
% for all cubes $Q$ in $\mathbb R^n$, where $p'$ denotes the conjugate exponent of $p$ defined by $p'=\frac{p}{p-1}$.  An $A_1$ weight is a weight that satisfies
% $$\left(\frac1{|Q|}\int_Q w(x) \,dx\right)\leq C(\inf_{Q} w)$$
% for some constant $C>0$ and all cubes $Q$ in $\mathbb R^n$.  Finally,
% we will define the class $A_\infty$ as the union over all $A_r$ for
% $r\geq 1$.  
%
Given $w\in A_\infty$, we define
\[
r_w = \inf\{r\in (1,\infty) : w\in A_r\}
\]
and for $0<p<\infty$ we define the critical index $s_w$ of $w$ by 
\[  s_w = \floor{n\Big(\frac {r_w}{p}-1\Big)}_+.  \]
%
%
% Let $0<p<\infty$ and $w\in A_{\infty}(\mathbb{R}^n)$.
% A bounded function $a$ is said to be a $(p,\infty,w)$-atom if $a$ is supported on some cube $Q$ and
% satisfies
% \begin{equation}
% \label{eq.atom}
% \|a\|_{L^{\infty}}
% \leq
% 1,
% \quad
% \int_{\mathbb{R}^n}
% x^\alpha a(x)dx=0
% \end{equation}
% for all $|\alpha| \le N$, where $N$ is a fixed integer and $N\ge s_w$.

Our first result gives the boundedness of $m$-CZOs into weighted Lebesgue spaces.

\begin{theorem}  \label{thm:1}
Given an integer $m\ge 1$, $0<p_1,\ldots,p_m<\infty$,  and $w_k\in A_\infty$, $1\le k\le m$,
let $T$ be an $m$-CZO associated to a kernel $K$ that satisfies
\eqref{eqn:kernelregularity} for $N$ such that
\begin{equation}
\label{eq.N1}
N\ge \max\bigg\{\floor{mn\Big(\frac{r_{w_k}}{p_k}-1\Big)}_+,1\le k\le m\bigg\}+(m-1)n.
\end{equation}
Then 
$$
T: H^{p_1}(w_1)\times \cdots \times H^{p_m}(w_m)\to L^{p}(\overline{w}),
$$
where $\overline{w} = \prod_{k=1}^{m}w_k^{\frac{p}{p_k}}$ and
\[
\frac1p=\frac1{p_1}+\cdots+\frac1{p_m}.
\]
\end{theorem}

Our second result gives boundedness of $m$-CZOs into weighted Hardy spaces.

\begin{theorem}
\label{thm:2}
Given $p,p_1,\ldots,p_m,\overline{w},w_1,\ldots,w_m$ and $T$ as in
Theorem \ref{thm:1}, suppose  the kernel $K$ satisfies
\eqref{eqn:kernelregularity} for $N$ such that 
\begin{equation}
\label{eq.N2}
N> s_{\overline{w}}+\max\bigg\{\floor{mn\Big(\frac{r_{w_k}}{p_k}-1\Big)}_+,1\le k\le m\bigg\}+mn.
\end{equation}
Suppose further that $T$ satisfies the cancellation condition \eqref{eq.can}
for all $|\alpha|\le s_{\overline{w}}$,
where for $1\leq k \leq m$,  $a_k$ is an $(N,\infty)$ atom:  i.e.,
$a_k$ is supported on a cube $Q_k$, $\|a_k\|_\infty\leq 1$, and 
\begin{equation}
\label{eq.atom}
\int_{\mathbb{R}^n}  x^\beta a_k(x)dx=0
\end{equation}
for all $|\beta|\le N$.
Then
$$
T: H^{p_1}(w_1)\times \cdots \times H^{p_m}(w_m)\to H^{p}(\overline{w}).
$$
\end{theorem}

\begin{remark}
In Theorems~\ref{thm:1} and~\ref{thm:2}, if all the weights $w_k=1$, then $r_{w_k}=1$, so we recapture the
unweighted results in~\cite{GK01,GNNS17A}.
\end{remark}

\begin{remark}
If $p>1$ and $w \in A_p$, then $H^p(w)=L^p(w)$ (see~\cite{ST89}).
Therefore, in Theorems~\ref{thm:1} and~\ref{thm:2},  if $w_k\in
A_{p_k}$, then we can replace $H^{p_k}(w_k)$ by $L^{p_k}(w_k)$ in the conclusion.
\end{remark}

\begin{remark}
Implicit in the statement of Theorem~\ref{thm:2} is the assumption
that $\overline{w} \in A_\infty$.  However, this is always the case:
see Lemma~\ref{lemma:Ainfty} below.
\end{remark}

\begin{remark}
Earlier, Xue and Yan \cite{XY12} proved a version of
Theorem~\ref{thm:1} with the additional  restriction that $0<p_k\le 1$
for all $1\le k\le m$.  We want to thank the authors for calling our
attention to their paper, which we had overlooked.
\end{remark}

\bigskip

Our next pair of results are the analogs of Theorems~\ref{thm:1}
and~\ref{thm:2} for the variable Lebesgue spaces.  The variable
Lebesgue spaces are a generalization of the classical $L^p$ spaces
with the exponent $p$ replaced by a measurable exponent function $\pp : \R^n \rightarrow
(0,\infty)$.  It consists of all measurable functions $f$ such that
for some $\lambda>0$. 
\[ \rho(f/\lambda) = \int_{\R^n}
  \left(\frac{|f(x)|}{\lambda}\right)^{p(x)}\,dx < \infty.  \]
This becomes a quasi-Banach space with quasi-norm
\[ \|f\|_\pp = \inf\{ \lambda >0 : \rho(f/\lambda)\leq 1 \}.  \]
If $p(x)\geq 1$ a.e., then this is a norm and $\Lp$ is a Banach
space.  These spaces were introduced by Orlicz~\cite{Orlicz} in 1931,
and have been extensively studied by a number of authors in the past
25 years.  For complete details and references, see~\cite{CF13}.
Variable Hardy spaces were introduced by the first author and
Wang~\cite{CW14} and independently by Nakai and Sawano~\cite{NS12}.

In variable Lebesgue exponent spaces,
harmonic analysis requires some assumption of regularity on the
exponent function $\pp$.  A common assumption that is sufficient for
almost all applications is that the exponent function is log-H\"older
continuous both locally and at infinity.  More precisely, there exist
constants $C_0$, $C_\infty$ and $p_\infty$ such that
\begin{equation} \label{eqn:log-holder1}
 |p(x)-p(y)| \leq \frac{C_0}{-\log(|x-y|)}, \qquad
  0<|x-y|<\frac{1}{2}, 
\end{equation}
and
\begin{equation} \label{eqn:log-holder2}
 |p(x) - p_\infty| \leq \frac{C_\infty}{\log(e+|x|)}. 
\end{equation}
Finally, given an exponent function $\pp$, we define
\[ p_- = \essinf_{x\in \R^n} p(x), \qquad p_+ = \esssup_{x\in \R^n}
  p(x). \]

As an immediate application of Theorems~\ref{thm:1} and~\ref{thm:2},
and multilinear Rubio de Francia extrapolation in the scale of
variable Lebesgue spaces, we get the following two results.

\begin{theorem}
\label{thm:3}
Given an integer $m\ge 1$, let $p_1,\ldots,p_m$ be real numbers, and
let  $q_1(\cdot),\ldots,q_m(\cdot)$ be
log-H\"older continuous
exponent functions such that $0<p_k<(q_k)_-\leq (q_k)_+<\infty$.  Define
\[
\frac1{q(\cdot)}
=
\frac1{q_1(\cdot)}
+\cdots
+
\frac1{q_m(\cdot)},
\qquad
\frac1{p}=\frac1{p_1}+\cdots+\frac1{p_m}.
\]
Let $T$ be an $m$-CZO as in Theorem \ref{thm:1} satisfying \eqref{eqn:kernelregularity} for all $|\alpha|\le N$ with
\[
N\ge \max\bigg\{\floor{mn\Big(\frac{1}{p_k}-1\Big)}_+,1\le k\le m\bigg\}+(m-1)n.
\]
Then 
$$
T: H^{q_1(\cdot)}\times \cdots \times H^{q_m(\cdot)}\to L^{q(\cdot)}.
$$
\end{theorem}

\begin{theorem}
\label{thm:4}
Given $q(\cdot),q_1(\cdot),\ldots,q_m(\cdot)$, $p,p_1,\ldots,p_m$ as in Theorem \ref{thm:3},
let $T$ be an $m$-CZO as in Theorem~\ref{thm:1}
satisfying \eqref{eqn:kernelregularity} for all $|\alpha|\le N$ with
\[
N> \floor{n\Big(\frac {1}{p}-1\Big)}_+
+\max\bigg\{\floor{mn\Big(\frac{1}{p_k}-1\Big)}_+,1\le k\le m\bigg\}+mn.
\]
Suppose further
that $T$ satisfies \eqref{eq.can} for all $|\alpha|\le
\floor{n(1/p-1)}_+$. Then 
$$
T: H^{q_1(\cdot)}\times \cdots \times H^{q_m(\cdot)}\to H^{q(\cdot)}.
$$
\end{theorem}

\begin{remark}
As we were completing this paper we learned that a version of 
Theorems~\ref{thm:3} and~\ref{thm:4},  with the additional hypothesis
that $(q_k)_+\le 1$ for all $1\le k\le m$,  was independently proved
by Tan~\cite{Tan17}.  We want to thank the author for sharing with us
a preprint of his work.
\end{remark}

The remainder of this paper is organized as follows.  In
Section~\ref{section:weights} we give some basic definitions and
theorems about weights that we will use in subsequent sections.  In
particular, we prove a finite atomic decomposition for weighted Hardy
spaces that extends the results in~\cite{CW14}.    In
Section~\ref{section:auxiliary} we gather together a number of
technical lemmas that we need for the proofs of Theorems~\ref{thm:1}
and~\ref{thm:2}.  Then in Sections~\ref{section:proof1}
and~\ref{section:proof2} we prove these results.  Finally, in
Section~\ref{section:variable} we give some basic facts about variable
exponent spaces and prove Theorems~\ref{thm:3}
and~\ref{thm:4}.  In fact, we prove more general results which include
these theorems as special cases.  Their statements, however, require
additional facts about variable exponent spaces, and so we delay their
statement until the final section.

Throughout this paper, we will use $n$ to denote the dimension of the
underlying space, $\R^n$, and will use $m$ to denote the ``dimension''
of our multilinear operators.  By a cube $Q$ we will always mean a cube
whose sides are parallel to the coordinate axes, and for $\tau>1$ let
$\tau Q$ denote the cube with same center such that $\ell(\tau Q)=\tau
\ell(Q)$. 
We define the
average of a function $f$ on a cube $Q$ by $f_Q= \avgint_Q f\,dx=
|Q|^{-1}\int_Q f\,dx$.   By $C$, $c$, etc. we will mean
constants that may depend on the underlying parameters in the
problem.  Sometimes, to emphasize that they (only) depend on certain
parameters, we will write $C(X,Y,Z,\ldots)$.  The values of these
constants may change from line to line.  If we write $A\lesssim B$, we
mean that $A\leq cB$ for some constant~$c$.

\section{Weights and weighted Hardy spaces}
\label{section:weights}

\subsection*{Weights and weighted norm inequalities}
In this section we give some basic definitions and results about
$A_p$ weights.  For complete information, we refer the reader
to~\cite{duoandikoetxea01,garcia-cuerva-rubiodefrancia85, grafakos08b}.
  By a weight~$w$ we always mean a non-negative, locally
integrable function such that $0<w(x)<\infty$ a.e.  For $1<p<\infty$,
we say that $w$ is in the Muckenhoupt class $A_p$, denoted by $w\in
A_p$, if 
\[ [w]_{A_p}=\sup_Q\left(\avgint_Q w\,dx\right)\left(\avgint_Q
    w^{1-p'}\,dx\right)^{p-1} <\infty.\]
When $p=1$ we say
that $w\in A_1$ if there is a constant $C$ such that for every
cube $Q$ and a.e. $x\in Q$,
\[  \avgint_Q w\,dx \leq C w(x). \]
The infimum over all such constants will be denoted by $[w]_{A_1}$.  The $A_p$ classes are nested:  for $1<p<q<\infty$, $A_1\subsetneq A_p
\subsetneq A_q$.  Let $A_\infty$ denote the union of all the $A_p$
classes, $p\geq 1$.

Given $w\in A_\infty$, then $w$ is a doubling measure.  More precisely,
if $w\in A_p$ for some $p\geq 1$, then it follows from the definition
that given any cube $Q$ and $\tau>1$, 
\[  w(\tau Q) \leq C\tau^{np} w(Q). \]

\medskip

In the study of multilinear weighted norm inequalities, we often need
the fact that the convex hull of $A_\infty$ weights is again in
$A_\infty$.  
The following result can be found, for instance, in~\cite{XY12} or 
in~\cite[Lemma~5]{GM04}.  For completeness we sketch a short
proof, using a multilinear reverse H\"older inequality:  if
$w_1,\ldots,w_m\in A_\infty$,  $1<p_1,\ldots,p_m<\infty$, and
$\frac{1}{p}=\frac{1}{p_1}+\cdots+\frac{1}{p_m}$, then for every cube $Q$,
\[ \prod_{k} \left(\avgint_Q w_k\,dx\right)^{\frac{p}{p_k}}
\lesssim
\avgint_Q \prod_k w_k^{\frac{p}{p_k}}\,dx. \]
This was originally proved in the bilinear case by the first author
and Neugebauer~\cite{cruz-uribe-neugebauer95}; for simpler proofs in
the multilinear case, see~\cite{DCU-KMP17,XY12}. 

\begin{lemma} \label{lemma:Ainfty}
Given $m\geq 1$, $1<p_1,\ldots,p_m<\infty$,
$\frac{1}{p}=\frac{1}{p_1}+\cdots+\frac{1}{p_m}$,  if
$w_1,\ldots,w_m \in A_\infty$, then  $\overline{w}=\prod_{k=1}^m
w_k^{\frac{p}{p_k}} \in A_\infty$. 
\end{lemma}

\begin{proof}
Since each $w_k \in A_\infty$, by choosing $C$ sufficiently large and
$\delta<1$ sufficiently close to $1$, we have that for every cube $Q$
and $E\subset Q$,
\[ \frac{w_k(E)}{w_k(Q)} \leq C\left(\frac{|E|}{|Q|}\right)^\delta. \]
But then, if we apply H\"older's inequality and the multilinear
reverse H\"older's inequality, we have that
\[ \frac{\overline{w}(E)}{\overline{w}(Q)} 
\leq
\frac{\prod_{k=1}^m \left(\int_E w_k\,dx\right)^{\frac{p}{p_k}}}
{\prod_{k=1}^m \left(\int_Q w_k\,dx\right)^{\frac{p}{p_k}}}
\leq
C\left(\frac{|E|}{|Q|}\right)^{\delta}. \]
\end{proof}

\medskip

There is a close connection between Muckenhoupt weights and the
Hardy-Littlewood maximal operator, defined by 
\[ Mf(x) = \sup_Q \avgint_Q |f(y)|\,dy \cdot \chi_Q(x), \]
where the supremum is taken over all cubes $Q$.   We have that if
$1<p<\infty$, then the maximal operator is bounded $L^p(w)$ if and
only if $w\in A_p$.  Moreover, we have a weighted vector-valued
inequality that generalizes the Fefferman-Stein inequality.  This was
first proved by Anderson and John~\cite{AJ80}; for an elementary proof
via extrapolation, see~\cite{CMP11}. 

\begin{lemma} \label{vv-max}
Given $1<p,q<\infty$ and $w\in A_p$, then for any sequence $\{f_k\}$ in
$L^p(w)$, 
\[ \Big\|\Big(\sum_k (Mf_k)^q\Big)^{\frac{1}{q}}\Big\|_{L^p(w)} 
\lesssim 
\Big\|\Big(\sum_k |f_k|^q\Big)^{\frac{1}{q}}\Big\|_{L^p(w)}. \]
\end{lemma}

\begin{remark} \label{FS-argument}
Below we will repeatedly apply Lemma~\ref{vv-max} in the following
way.  Fix $0<p<\infty$ and $w\in A_\infty$.   Then $w\in A_q$ and
without loss of generality we may assume $p<q$. Let
$r=\frac{q}{p}>1$.    Given a sequence of cubes $Q_k$, let
$Q_k^*=\tau Q_k$, $\tau>1$.  Then $\chi_{Q_k^*} \lesssim M(\chi_{Q_k})$, and the
implicit constant depends only on $n$ and $\tau$.   But then by
Lemma~\ref{vv-max}, we have that for any non-negative $\lambda_k$, 
\begin{multline*}
\Big\| \sum_k \lambda_k\chi_{Q_k^*}\Big\|_{L^p(w)}
 \lesssim
\Big\| \sum_k M\big(\lambda_k ^{\frac{1}{r}}\chi_{Q_k}^{\frac{1}{r}}\big)^r\Big\|_{L^p(w)}
 = 
\Big\| \Big(\sum_k
  M\big(\lambda_k ^{\frac{1}{r}}\chi_{Q_k}^{\frac{1}{r}}\big)^r\Big)^{\frac{1}{r}}
\Big\|_{L^q(w)}^r \\
 \lesssim 
\Big\| \Big(\sum_k
\lambda_k\chi_{Q_k}\Big)^{\frac{1}{r}}
\Big\|_{L^q(w)}^r 
 = 
\Big\| \sum_k \lambda_k\chi_{Q_k}\Big\|_{L^p(w)}.
\end{multline*}
\end{remark}

\medskip

Below we will need to prove a weighted norm inequality for an
$m$-CZO.  To do so, we will make use of some recent developments in the theory
of harmonic analysis on the domination of singular integrals by sparse
operators.    Here we sketch the basic definitions; for further
information, see, for instance,~\cite{CruzUribe:2016ji}.  

A collection of cubes $\Sp$ is called a {\it sparse family} if each cube
$Q\in \Sp$ contains measurable subset $E_Q\subset Q$ such that
$|E_Q|\geq \frac12|Q|$ and the family $\{E_Q\}_{Q\in \Sp}$ is pairwise
disjoint.  Given a sparse family $\Sp$ we define a linear sparse
operator
\[ T^\Sp f(x) = \sum_{Q\in \Sp} \avgint_Q f(y)\,dy\cdot \chi_Q(x). \]
The following estimate is proved in~\cite{CMP12, M12}.

\begin{proposition} \label{prop:linear-sparse}
If $1<q<\infty$ and $w\in A_q$, then given any sparse linear operator
$T^\Sp$,
$$
\|
T^\Sp f
\|_{L^q(w)}
=
\Big\| \sum_{Q\in \Sp}\avgint_Q
    f\,dy\cdot \chi_Q\Big\|_{L^q(w)}\leq
C[w]_{A_q}^{\max(1,\frac{1}{q-1})}\|f\|_{L^q(w)}.$$
\end{proposition}

In a similar way, given a sparse family $\Sp$ we define the 
multilinear sparse operator
$$T^\Sp(f_1,\ldots,f_m)(x)=\sum_{Q\in \Sp}
\prod_{k=1}^m
\avgint_Q  f_k(y_k)\,dy_k\cdot \chi_Q(x).$$ 
The following pointwise domination theorem was proved in \cite[Theorem
13.2]{LN15} (see also \cite{CR16}).

\begin{proposition}\label{thm:LernNaz} 
Let $T$ be an $m$-CZO whose kernel satisfies
\eqref{eqn:kernelregularity} for any $N\geq 1$.  
Then given any collection $f_1,\ldots,f_m$ of bounded functions of
compact support,  there exists $3^n$ sparse families $S_j$ such that
$$|T(f_1,\ldots,f_m)(x)|\leq C\sum_{j=1}^{3^n}
T^{\Sp_j}(|f_1|,\ldots,|f_m|)(x).$$
%x
\end{proposition}

\subsection*{Weighted Hardy spaces}
In this section we define the weighted Hardy spaces and prove a finite
atomic decomposition theorem.  In defining them we follow Str\"omberg
and Torchinsky~\cite{ST89} and we refer the reader there for more
information.

Let $\mathscr{S}(\mathbb{R}^n)$ denote the Schwartz class of smooth
functions.  For $N_0\in \mathbb N$ to be a large value determined later, define
\[
\mathfrak{F}_{N_0} 
= 
\{ \varphi\in\mathscr{S}(\mathbb{R}^n) : 
\int (1+|x|)^{N_0}\Big( \sum_{|\alpha|\le N_0}
\Big| \frac{\partial^{\alpha}}{\partial x^{\alpha}}\varphi(x) \Big|^2 \Big)dx\le 1 \}.
\]
Fix  $0<p<\infty$ and $w\in A_{\infty}$; we define the weighted Hardy
space $H^p(w)$ to be the set of distributions
\[
H^p(w) = \{ f\in \mathscr{S}'(\mathbb{R}^n) : \mathcal{M}_{N_0}(f)\in L^p(w) \}
\]
with the quasi-norm
\[
\|f\|_{H^p(w)} = \|\mathcal{M}_{N_0}(f)\|_{L^p(w)},
\]
where the grand maximal function $\mathcal{M}_{N_0}(f)$ is defined by
\[
\mathcal{M}_{N_0}(f)(x) = \sup_{\varphi\in\mathfrak{F}_{N_0}}\sup_{t>0}\big|
\varphi_t*f(x)
\big|.
\]
Note that in this definition, $N_0$ is taken to be a large positive
integer, depending on $n,p$ and $w$, whose value is chosen so that the
usual definitions of unweighted Hardy spaces remain equivalent in the
weighted setting.  Its exact value does not matter for us.  

Given an integer $N>0$, an $(N,\infty)$ atom is a function $a$ such
that there exists a cube $Q$ with $\supp(a)\subset Q$, $\|a\|_\infty
\leq 1$, and for  $|\beta| \leq N$,
\[ \int_{\R^n} x^\beta a(x)\,dx = 0. \]
In~\cite[Chapter VIII]{ST89} it was shown that every $f\in H^p(w)$ has
an atomic decomposition: for every $N\geq s_w$ there exist a sequence
of non-negative numbers $\{\lambda_k\}$ and a sequence of smooth
$(N,\infty)$ atoms $\{a_k\}$ with $\supp(a_k)\subset Q_k$,  such that
\[
f = \sum_{k}\lambda_ka_k,
\]
and the sum converges in the sense of distributions and in the
$H^p(w)$ quasi-norm.   Moreover, we have that
\[ \Big\|\sum_k \lambda_k \chi_{Q_k}\Big\|_{L^p(w)}
\lesssim
\|f\|_{H^p(w)}. \]

Below, we want to use the atomic decomposition to estimate the norm of
an $m$-CZO.  One technical obstacle, however, is that this atomic
decomposition may be an infinite sum, and therefore it is not
immediate that we can exchange sum and integral in the definition of
an $m$-CZO.  For the argument to overcome this problem in the
unweighted setting, see~\cite{GH14}. Our approach here is different:
we show that for a dense subset of $H^p(w)$, we can form the atomic
decomposition using a finite sequence of atoms.  Our result
generalizes a result in the unweighted case from~\cite{MSV08}; in the
weighted case it generalizes results proved in~\cite{CW14,NS12}.

To state our result, note that for  $N\ge s_w$, if  we define
\[
\mathcal{O}_N
=
\big\{f\in \mathcal{C}_0^{\infty}\ :\
\int_{\mathbb{R}^n}x^{\alpha}f(x)\,dx=0,\quad 0 \leq|\alpha|\le N \big\},
\]
then $\mathcal{O}_N\cap H^p(w)$ is dense in $H^p(w)$. 

\begin{theorem} \label{lm.fin.dec} Fix $w\in A_\infty$ and
  $0<p<\infty$, and let $N\geq s_w$.  For each
  $f\in \mathcal{O}_N\cap H^p(w)$, there exists a
  finite sequence of non-negative numbers $\{\lambda_k\}_{k}$ and a
  sequence  $\{a_k\}$ of $(N,\infty)$ atoms, $\supp(a_k)\subset Q_k$, such that
  $f=\sum_k\lambda_k a_k$ and
\begin{equation}
\label{eq.dec.norm}
\Big\|
\sum_{k}\lambda_k\chi_{Q_k}
\Big\|_{L^{p}(w)}
\le C\|f\|_{H^p(w)},
\end{equation}
\end{theorem}

The proof of Theorem~\ref{lm.fin.dec} is gotten by a close analysis of
the atomic decomposition given above.  To prove it, we use the
following technical result.  It is adapted from the corresponding
result from \cite[Chapter VIII]{ST89} (in the weighted case) and from
the proof of the unweighted version of Theorem~\ref{lm.fin.dec}
in~\cite{MSV08}. (See also the construction of the atomic
decomposition in~\cite{CW14}.)  Indeed weights play almost no role in the result
except in $(4)$.  

\begin{lemma}
\label{lm.dec.ST}
Fix $w\in A_{\infty}$, $0<p<\infty$, and $N\geq s_w$,  and let
$f\in \mathcal{O}_N\cap H^p(w)$.   For each $k\in \mathbb{Z}$, let 
\[
\Omega_k = \{x\in \mathbb{R}^n : \mathcal{M}_{N_0}f(x)>2^k\}.
\]
Then there exists a sequence
$\{\beta_{k,i}\}$ of smooth functions with compact support and a
family of cubes $\{Q_{k,i}\}$ with finite overlap that
such that the following hold:
\begin{enumerate}
\item  For each $k$ and all $i$, $Q_{k,i}\subset Q_{k,i}^* \subset
  \Omega_k$, where $Q_{k,i}^*=\tau Q_{k,i}$ for a fix constant
  $\tau>1$ and the $Q_{k,i}^*$ also have finite overlap.

\item The $\beta_{k,i}$ are $(N,\infty)$ atoms with $\supp(\beta_{k,i})
  \subset Q_{k,i}^*$.  In particular,
  $\sum_{i}|\beta_{k,i}|\lesssim C$ uniformly for all
  $k\in \mathbb{Z}$.

\item $f=\sum_{k,i}\lambda_{k,i}\beta_{k,i}$, where the convergence is
  unconditional  both pointwise and in the sense of distributions.

\item $\lambda_{k,i}\lesssim 2^k$ for all $k,i$ and
  $\sum_{k,i}\lambda_{k,i}\chi_{Q_{k,i}}\lesssim \mathcal{M}_{N_0}(f)$.
  In particular,  $\sum_{k,i}\lambda_{k,i}\beta_{k,i}$ also converges absolutely
  to $f$ in $L^q(w),$  whenever $q>1$ is such that $w\in A_q$.
\end{enumerate}
\end{lemma}

\begin{proof}[Proof of Theorem~\ref{lm.fin.dec}]
Fix $f\in \mathcal{O}_N \cap H^p(w)$; by homogeneity we may assume
without loss of generality that  $\|f\|_{H^p(w)} =1$. Then there
exists $R>1$ such that $\supp(f)\subset B(0,R) = B$.
Let $B^* = B(0,4R)$.  We claim that for all $x\notin B^*$, 
\begin{equation}
\label{eq.mnb}
\mathcal{M}_{N_0}f(x)
\lesssim w(B)^{\frac{-1}{p}}\|f\|_{H^p(w)}
 \lesssim\frac{1}{w(B^*)^{\frac{1}{p}}}.
\end{equation}
To prove this, we argue as in~\cite[Lemma~7.11]{CW14} (cf. inequality (7.7)).   There they
showed a pointwise inequality:  given any $\varphi\in \mathfrak{F}_{N_0}$ and $t>0$,
\[ |f*\varphi_t(x)| \lesssim \inf_{z\in B_*} \mathcal{M}_{N_0} f(z), \]
where $B_*=B(0,\frac{1}{2}R)$.   Therefore, we have that
\[ |f*\varphi_t(x)|^p \lesssim  \frac{1}{w(B_*)}\int_{B_*}
  \mathcal{M}_{N_0}f(z)^p w(z)\,dz  \leq \frac{1}{w(B_*)}; \]
inequality~\eqref{eq.mnb} follows if we take the supremum over all
$\varphi\in \mathfrak{F}_{N_0}$ and $t>0$, and note that since
$w\in A_\infty$, $w(B^*) \lesssim w(B_*)$. 

Now let $k_0$ be the smallest integer such that for all $k>k_0$,
$\Omega_k\subset B^*$. More precisely, by \eqref{eq.mnb} we can take
$k_0$ to be the largest integer such that
$2^{k_0}\le Cw(B^*)^{\frac{-1}p}$.

\medskip

By Lemma \ref{lm.dec.ST} we can decompose $f$ as
\[
f = \sum_{k,i}\lambda_{k,i}\beta_{k,i}
\]
where the $\beta_{k,i}$ are $(N,\infty)$ atoms.  We will show that this
sum can be rewritten as a finite sum of atoms. Set
\[
F_1 = \sum_{k\le k_0}\sum_i\lambda_{k,i}\beta_{k,i} = f-\sum_{k> k_0}\sum_i\lambda_{k,i}\beta_{k,i}.
\]
Since the $\beta_{k,i}$ are supported in $\Omega_k\subset B^*$ for all $k>k_0,$ the function $F_1$ is also supported in $B^*$. Moreover
\[
\|F_1\|_{\infty}\le \sum_{k\le k_0}
\Big\|\sum_i\lambda_{k,i}|\beta_{k,i}|\Big\|_{L^{\infty}}
\lesssim \sum_{k\le k_0}2^k
= C_12^{k_0}.
\]
Further, $F_1$ has vanishing moments up to order $N$. To see this, fix
$|\alpha|\leq N$ and $q>1$ such that $w\in A_q$.  Then, since $\supp(\beta_{k,i})\subset B^*$,
\begin{multline*}
 \Big\| \sum_{k\leq k_0} \sum_i |x^\alpha|
  |\lambda_{k,i}\beta_{k,i}| \Big\|_{L^1} 
\leq 
(4R)^{|\alpha|} \Big\| \sum_{k\leq k_0} \sum_i 
  |\lambda_{k,i}\beta_{k,i}| \Big\|_{L^q(w)}
  w^{1-q'}(B^*)^{\frac{1}{q'}} \\
\lesssim 
(4R)^{|\alpha|} \|\mathcal{M}_{N_0}f\|_L^q(w) w^{1-q'}(B^*)^{\frac{1}{q'}}
\lesssim 
(4R)^{|\alpha|} \|f\|_L^q(w) w^{1-q'}(B^*)^{\frac{1}{q'}} < \infty. 
\end{multline*}
Therefore, the series on the left-hand side converges absolutely, so
you can exchange the sum and integral; since each $\beta_{k,i}$ has
vanishing moments, so does $F_1$.   Therefore, if we set $a_{0} = C_1^{-1}2^{-k_0}F_1$ then $a_{0}$ is an $(N,\infty)$ atom supported in $B^*$.

To estimate the remaining terms, note that $f$ is a bounded function
and so there exists an integer $k_\infty>k_0$ such that
$\Omega_k=\emptyset$ for all $k\ge k_\infty$. Thus the sum
\[
\sum_{k> k_0}\sum_i\lambda_{k,i}\beta_{k,i} = \sum_{k_0<k< k_\infty}\sum_i\lambda_{k,i}\beta_{k,i}
\]
has finite many terms under the summation of $k$ indices.  Further,
since the sum
$\sum_{k,i} \lambda_{k,i}\chi_{Q_{k,i}} \lesssim \mathcal{M}_{N_0}f$ it
converges everywhere.  Therefore, for each $k_0<k< k_\infty$ there
exists an integer $\rho_k$ such that
\[
\sum_{i>\rho_k}\lambda_{k,i}|\beta_{k,i}|\le 2^{-k_\infty}w(B^*)^{-\frac1p}.
\]

If we define
\[
F_2 = \sum_{k_0<k< k_\infty}\sum_{i>\rho_k}\lambda_{k,i}\beta_{k,i},
\]
then $F_2$ is supported in $B^*$ and
\[
\|F_2\|_{\infty}\le \sum_{k_0<k< k_\infty}2^{-k_\infty}w(B^*)^{-\frac1p}
\leq C_2 w(B^*)^{-\frac1p}.
\]
Moreover, arguing as we did above for $F_1$, we have that $F_2$ has
vanishing moments for $|\alpha|\leq N$.  
Thus if we set $a_{\infty} = C_2^{-1}w(B^*)^{\frac1p}F_2$, then
$a_{\infty}$ is an $(N,\infty)$ atom.

Therefore, we have shown that we can decompose $f$ as a finite sum of
$(N,\infty)$ atoms:
\begin{equation}
\label{eq.fin.dec}
f = (C_12^{k_0})a_{0}+\sum_{k_0<k<k_\infty}\sum_{1\le i\le \rho_k}
\lambda_{k,i}\beta_{k,i}
+C_2w(B^*)^{-\frac1p}a_{\infty}.
\end{equation}

It remains to prove that~\eqref{eq.dec.norm} holds.   But by our
choice of $k_0$, we have that $\|C_12^{k_0}\chi_{B^*}\|_{L^p(w)}\le
C$, and clearly $\|w(B^*)^{-\frac1p}\chi_{B^*}\|_{L^p(w)}\le C$.  
Finally, by the weighted Fefferman-Stein inequality (see
Remark~\ref{FS-argument}), we have that
\[
\Big\|\sum_{k_0<k<k_\infty}\sum_{1\le i\le \rho_k}
\lambda_{k,i}\chi_{Q_{k,i}^*}\Big\|_{L^p(w)}
\lesssim
\Big\|\sum_{k_0<k<k_\infty}\sum_{1\le i\le \rho_k}
\lambda_{k,i}\chi_{Q_{k,i}}\Big\|_{L^p(w)}
\lesssim
\|\mathcal{M}_{N_0}f \|_{L^p(w)}\lesssim 1.
\]
Since $\|f\|_{H^p(w)}=1$, we get the desired inequality, and  this completes the proof of Theorem~\ref{lm.fin.dec}.
\end{proof}

\section{Auxiliary results}
\label{section:auxiliary}

%%%%%%%%%%%%%%%%%%%%%%%%%%%%%%%%%%

In this section we state and prove several lemmas on averaging
operators and $m$-CZOs needed for the
proofs of Theorems~\ref{thm:1} and~\ref{thm:2}.

\subsection*{Averaging operators}

We begin with a well-known result on the maximal operator $M_\mu$
defined with respect to a measure $\mu$:

\[ M_\mu f(x) = \sup_Q \frac{1}{\mu(Q)}\int_Q |f|\,d\mu \cdot
  \chi_Q(x). \]
For a proof, see~\cite[Chapter~II]{garcia-cuerva-rubiodefrancia85}.

\begin{proposition} \label{prop:mu-max}
Let $\mu$ be a doubling measure on $\R^n$.  Then the maximal operator
$M_\mu$ satisfies the weak $(1,1)$ inequality
\begin{equation} \label{eqn:mu-max1}
 \sup_{t>0} t\, \mu(\{ x \in \R^n : M_\mu f(x) > t \}) \leq
  C(\mu) \int_\subRn |f|\,d\mu, 
\end{equation}
and for $1<p<\infty$ the strong $(p,p)$ inequality
\begin{equation} \label{eqn:mu-max2}
 \int_\subRn (M_\mu f)^p\,d\mu \leq C(\mu,p)\int_\subRn
  |f|^p\,d\mu. 
\end{equation}
\end{proposition}

The next three lemmas on averaging operators are weighted extensions
of results from ~\cite{GK01}.  Our proofs, however, are different and
are motivated by ideas from~\cite{ST89}.

\begin{lemma}
\label{lm.Wp1}
Let $\mu$ be a doubling measure on $\mathbb{R}^n$ and fix $0<p<1$.
Then given any finite collection $\mathcal J$ of cubes and any set    $\{f_Q:\,\, Q\in
\mathcal J\}$   of non-negative integrable functions 
with $\supp(f_Q) \subset Q$, 
$$
\Big\|\sum_{Q\in\mathcal{J}}f_Q\Big\|_{L^p(\mu)} \le \,
C(\mu,p,n)\,  \Big\|\sum_{Q\in\mathcal {J }}a_1^\mu(Q)\chi_Q \Big\|_{L^p(\mu)},
$$
where
$$
a_1^\mu(Q)=\mu(Q)^{-1}\int_Q f_Q(x)\, d\mu(x).
$$
\end{lemma}

\begin{proof}
Let $F=\sum_{Q\in\mathcal{J}}f_Q$ and $G=\sum_{Q\in\mathcal {J
  }}a_1^\mu(Q)\chi_Q$ and for each $t>0$ let
\[
L_t = \{x\in \mathbb{R}^n : G(x)>t\}, \qquad 
U_t = \{y\in \mathbb{R}^n : M_\mu\chi_{L_t}(y)>\frac14\}.
\]
By \eqref{eqn:mu-max1} we have that
$\mu(U_t)\le C(\mu)\mu(L_t)$.
We can now estimate as follows:
\begin{align*}
\mu(\{x\in \mathbb{R}^n : F(x)>t\})
\le& \mu(U_t)+ \mu({U_t^c\cap \{x\in \mathbb{R}^n : F(x)>t\}})\\
\lesssim& \mu(L_t)+ \frac1t\int_{U_t^c}F(x)\,d\mu(x)\\
\lesssim& \mu(L_t)+ \frac1t\sum_{Q\in \mathcal{J}:Q\cap U^c_t\ne\emptyset}\int_{Q}f_Q(x)\,d\mu(x)\\
\lesssim& \mu(L_t)+ 
\frac1t\sum_{Q\in \mathcal{J}:Q\cap U^c_t\ne\emptyset}
a^\mu_1(Q)\mu(Q).
\end{align*}
If $Q\in \mathcal{J}$ is such that $Q\cap U^c_t\ne\emptyset$, then
$M_\mu\chi_{L_t}(z)\le \frac14$ for all $z\in Q\cap U^c_t$. In particular, we have
\[
\frac{\mu(L_t\cap Q)}{\mu(Q)}\le \frac14,
\]
and so $\mu(Q)\le \frac43\mu(L^c_t\cap Q)$ for all $Q\in \mathcal{J}$.
Thus we have that
\begin{align*}
\mu(\{x\in \mathbb{R}^n : F(x)>t\})
\lesssim&
\mu(L_t)+ 
\frac1t\sum_{Q\in \mathcal{J}}
a^\mu_1(Q)\mu(Q\cap L^c_t)\\
\lesssim&
\mu(L_t)+ 
\frac1t\sum_{Q\in \mathcal{J}}
a^\mu_1(Q)\int_{L^c_t}\chi_Q(x)\,d\mu(x)\\
\lesssim&
\mu(L_t)+ 
\frac1t\int_{L^c_t}G(x)\,d\mu(x).
\end{align*}

Given this estimate, if we 
multiply by $pt^{p-1}$ and integrate, by Fubini's theorem we get
\begin{align*}
\|F\|_{L^p(\mu)}^p
& = \int_0^\infty pt^{p-1}\mu(\{ x\in \R^n : F(x) >t \})\,dt \\
& \lesssim \int_0^\infty pt^{p-1}\mu(\{ x\in \R^n : G(x) >t \})\,dt 
 + \int_0^\infty pt^{p-2} \int_{\{ x\in \R^n : G(x)\leq t\}}
  G(x)\,d\mu(x) \\
& = \int_\subRn G(x)^p \,d\mu(x) 
+ \int_\subRn G(x) \int_{G(x)}^\infty pt^{p-2}\,dt \,d\mu(x) \\
& \lesssim \|G\|_{L^p(\mu)}^p.
\end{align*}
\end{proof}

\begin{lemma}
\label{lm.Wpq}
Let $\mu$ be a doubling measure on $\mathbb{R}^n$ and fix $1\leq
p<q<\infty$.  Then given any finite collection of cubes  $\mathcal{J}$
and any set    $\{f_Q:\,\, Q\in \mathcal J\}$   of non-negative integrable functions 
with $\supp(f_Q) \subset Q$,
\begin{equation}
\label{eq.wpq}
\Big\|\sum_{Q\in\mathcal{J}}f_Q\Big\|_{L^p(\mu)} \le \,
C(\mu,p,q,n)\,  \Big\|\sum_{Q\in\mathcal {J }}a_q^\mu(Q)\chi_Q \Big\|_{L^p(\mu)},
\end{equation}
where
$$
a_q^\mu(Q)=
\Big(
\frac{1}{\mu(Q)}
\int_Q |f_Q(x)|^q\, d\mu(x)
\Big)^{\frac1q}.
$$
\end{lemma}

\begin{proof}
First suppose that $p>1$; we estimate by duality.  Then there exists
non-negative  $g\in L^{p'}(\mu)$, $\|g\|_{L^{p'}(d\mu)} = 1$,  such
that 
\begin{align*}
\Big\|\sum_{Q\in\mathcal{J}}f_Q\Big\|_{L^p(\mu)}
& = \sum_{Q\in\mathcal{J}}\int_Q f_Q(x)g(x)\, d\mu(x)\\
& \le
\sum_{Q\in\mathcal{J}}
\Big( \int_Qf_Q(x)^q\, d\mu(x) \Big)^{\frac1q}
\Big( \int_Qg(x)^{q'}\, d\mu(x) \Big)^{\frac1{q'}}\\
&=\sum_{Q\in\mathcal{J}}
a^\mu_q(Q)\mu(Q)
\Big[\frac1{\mu(Q)}\int_Qg^{q'}d\mu \Big]^{\frac1{q'}}\\
&\le \sum_{Q\in\mathcal{J}}
a^\mu_q(Q)\int_{Q}
M_{\mu}(g^{q'})(x)^{\frac1{q'}}
\, d\mu(x)\\
&\le \int_\subRn
\Big[\sum_{Q\in\mathcal{J}}
a^\mu_q(Q)\chi_Q\Big]
M_{\mu}(g^{q'})(x)^{\frac1{q'}}
\, d\mu(x)\\
&\le
\Big\|\sum_{Q\in\mathcal{J}}
a^\mu_q(Q)\chi_Q \Big\|_{L^p(\mu)}
\|M_{\mu}(g^{q'})^{\frac1{q'}} \|_{L^{p'}(\mu)}\\
&=
\Big\|\sum_{Q\in\mathcal{J}}
a^\mu_q(Q)\chi_Q \Big\|_{L^p(d\mu)}
\|M_{\mu}(g^{q'})\|_{L^{\frac{p'}{q'}}(\mu)}^{\frac1{q'}} \\
& \lesssim 
\Big\|\sum_{Q\in\mathcal{J}}
a^\mu_q(Q)\chi_Q \Big\|_{L^p(\mu)};
\end{align*}
the first and third inequalities follow from H\"older's inequality, and the last
from~\eqref{eqn:mu-max2} (since $p'>q'$) and the fact that
$\|g\|_{L^{p'}(\mu)}=1$.

Finally, when $p=1$ the proof is essentially the same except that we
use use the fact that $M_\mu$ is bounded on $L^\infty$.  This
completes the proof.
\end{proof}

\begin{lemma}\label{lm:AvgLpw}
Let $w\in A_{\infty}$,  and fix  $0<p<\infty$ and $\max(1,p)<q<\infty$.
Then given any  collection
of cubes $\{Q_k\}_{k=1}^{\infty}$ 
and  nonnegative integrable functions
$\{g_k\}$ with  $\supp(g_k)\subset Q_k $,
$$
\Big\|\sum_{k=1}^{\infty} g_k \Big\|_{L^p(w)} \le C(w,p,q,n)\,  \Big\|\sum_{k=1}^{\infty} \Big(
 \frac{1}{w(Q_k) }\int_{Q_k} g_k(x)^qw(x)\, dx
\Big)^{\frac1q} \chi_{Q_k}  \Big\|_{L^p(w)} .
$$
\end{lemma}

\begin{proof}
Since $w\in A_{\infty}$, the measure $\mu = w(x)\,dx$ is doubling.  If
$p \geq 1$, then if we 
fix an arbitrary integer $K$ and apply Lemma~\ref{lm.Wpq} to the functions
$\{g_k\}_{k=1}^K$, we immediately get
$$
\big\| \sum_{k=1}^K g_k \big\|_{L^p(w)} \le  C(w,p,q,n)
\Big\| \sum_{k=1}^K 
\Big(\frac1{w(Q_{k})} \int_{Q_{k}} g_k(x)^qw(x)\, dx\Big)^{\frac1q}
\chi_{Q_k}  \Big\|_{L^p(w)}.  
$$
The desired inequality now follows from Fatou's lemma.

When $0<p<1$, we can apply Lemma~\ref{lm.Wp1} to get the same
conclusion, using the fact that 
\[
\frac1{w(Q_{k})} \int_{Q_{k}} g_k(x)w(x)\, dx  \leq
\bigg(\frac1{w(Q_{k})} \int_{Q_{k}} g_k(x)^qw(x)\, dx\bigg)^{\frac1q}.
\]
\end{proof}

\subsection*{Estimates for $m$-CZOs}
In this section we prove three estimates on $m$-CZOs.

\begin{lemma}\label{lm:TL2w}
Let $T$ be the operator as in Theorem \ref{thm:1} and fix $w\in
A_{q}$, $q>1$. Then given any collection $f_1,\ldots,f_m$ of bounded
functions of compact support,
\[
\|T(f_1,f_2,\ldots,f_m)\|_{L^q(w)}\le C\|f_1\|_{L^q(w)}\|f_2\|_{L^\infty}\cdots \|f_m\|_{L^\infty}.
\]
\end{lemma}

\begin{proof} 
  By the domination estimate in Proposition~\ref{thm:LernNaz} it will
  suffice to prove this estimate for any multilinear sparse operator
  $T^\Sp$ and non-negative functions $f_1,\ldots,f_m$.  By the definition of the
  sparse operator we have
\[ T^\Sp(f_1,\ldots,f_m)\leq \|f_2\|_\infty\cdots\|f_m\|_\infty 
\sum_{Q\in \Sp}\avgint_Q f_1\,dy \cdot \chi_Q 
= \|f_2\|_\infty\cdots\|f_m\|_\infty T^\Sp f_1, \]
where on the right-hand side we now have a linear sparse operator.
But then by Proposition~\ref{prop:linear-sparse} we have that
$$\|T^\Sp(f_1,\cdots,f_m)\|_{L^{q}(w)}
\lesssim \|T^\Sp f_1\|_{L^q(w)}\|f_2\|_\infty\cdots\|f_m\|_\infty
\lesssim \|f_1\|_{L^q(w)}\|f_2\|_\infty\cdots\|f_m\|_\infty.$$
\end{proof}

The following lemma was first prove in \cite{GNNS17A}.  For
completeness we include its short proof.

\begin{lemma}
\label{lm.4A1}
For $1\le k\le m$ let $a_k$ be an $(N,\infty)$ atom supported in $Q_k$
and let $c_k$ be the center of $Q_k$.  Then, given any non-empty
subset $\Lambda \subset \{1,\ldots,m\}$, we have that for all $y\notin \cup_{k\in \Lambda}Q_k^*$,
\begin{equation}
\label{eq.4A1}
|T(a_1,\ldots,a_m)(y)|
\lesssim
\dfrac{\min\{\ell(Q_k) : k\in \Lambda\}^{n+N+1}}
{\big(\sum_{k\in \Lambda}|y-c_k|\big)^{n+N+1}}.
\end{equation}
In particular, we always have that
\begin{equation}
\label{eq.4AA}
|T(a_1,\ldots,a_m)|\chi_{(Q^*_1\cap \ldots\cap Q_m^*)^c}
\lesssim
\prod_{k=1}^m\big(M(\chi_{Q_k})\big)^\frac{n+N+1}{mn}.
\end{equation}
\end{lemma}

\begin{proof}
Without loss of generality we may assume that $\Lambda=\{1,\ldots,r\}$ for some $1\le r\le m$ and that
\[
\ell(Q_1) = \min\{\ell(Q_k) : k\in \Lambda\}.
\]
Fix $y\notin \cup_{k\in \Lambda}Q_k^*$; because $a_1$ has vanishing
moments up to order $N$, we can rewrite
\begin{align}
\notag
T(a_1,\ldots,a_m)(y) =&
 \int_{\mathbb{R}^{mn}}K(y,y_1,\ldots,y_m)a_1(y_1)\cdots a_m(y_m)d\vec{y}\\
\notag
 =&
 \int_{\mathbb{R}^{mn}}\big[K(y,y_1,\ldots,y_m)-P_N(y,y_1,y_2,\ldots,y_m)\big]
 a_1(y_1)\cdots a_m(y_m)d\vec{y}\\
  \label{eq.4A2}
 =&
 \int_{\mathbb{R}^{mn}}K^1(y,y_1,y_2,\ldots,y_m)
 a_1(y_1)\cdots a_m(y_m)d\vec{y},
\end{align}
where 
\begin{equation*}
P_N(y,y_1,y_2,\ldots,y_m) = \sum_{|\alpha|\le N}\frac{1}{\alpha!}
\partial^{\alpha}_{2}K(y,c_1,y_2,\ldots,y_m)(y_1-c_1)^{\alpha}
\end{equation*}
is the Taylor polynomial of degree $N$ of $K(y,\cdot,y_2,\ldots,y_m)$ at $c_1$ and
\begin{equation}
\label{eq.4A3}
K^1(y,y_1,\ldots,y_m) = K(y,y_1,\ldots,y_m)-P_N(y,y_1,y_2,\ldots,y_m).
\end{equation}

By the smoothness condition of the kernel and the fact that
$|y-y_k|\approx |y-c_k|$
for all $k \in\Lambda$ and $y_k\in Q_k$
we have that
\begin{align*}
\big|K(y,y_1,\ldots,y_m)-P_N(y,c_1,y_2,\ldots,y_m)\big|\lesssim&
|y_1-c_1|^{N+1}\Big(
\sum_{k\in \Lambda}|y-c_k|+\sum_{j=2}^m|y-y_j|
\Big)^{-mn-N-1}.
\end{align*}
Thus,
\begin{align*}
|T(a_1,\ldots,a_m)(y)| \lesssim&
\int_{\mathbb{R}^{mn}}\frac{|y_1-c_1|^{N+1}|a_1(y_1)|\cdots |a_m(y_m)|}
 {\Big(
\sum_{k\in \Lambda}|y-c_k|+\sum_{j=2}^m|y-y_j|
\Big)^{mn+N+1}}
 d\vec{y}\\
 \lesssim&
 \int_{\mathbb{R}^{(m-1)n}}\frac{
 \ell(Q_1)^{n+N+1}}
 {\Big(
\sum_{k\in \Lambda}|y-c_k|+\sum_{j=2}^m|y_j|
\Big)^{mn+N+1}}
 dy_2\cdots dy_m\\
 \lesssim&
 \frac{\ell(Q_1)^{n+N+1}}
 {\Big(
\sum_{k\in \Lambda}|y-c_k|
\Big)^{n+N+1}},
\end{align*}
which implies \eqref{eq.4A1}.

To prove \eqref{eq.4AA},  fix $y\in (Q_1^*\cap\ldots\cap Q_m^*)^c$;
then there exists a non-empty subset $\Lambda$ of $\{1,\ldots,m\}$
such that $y\notin Q_k^*$ for all $k\in \Lambda$ and $y\in Q_l^*$ for
$l\notin \Lambda.$ Then by~\eqref{eq.4A1} we have that
\begin{multline*}
|T(a_1,\ldots,a_m)(y)|
\lesssim
\dfrac{\min\{\ell(Q_k) : k\in \Lambda\}^{n+N+1}}
{\big(\sum_{k\in \Lambda}|y-c_k|\big)^{n+N+1}}\\
\lesssim
\prod_{k\in\Lambda}
\Big(
\frac{\ell(Q_k)}{\ell(Q_k)+|y-c_k|}
\Big)^{\frac{n+N+1}{mn}}
\lesssim
\prod_{k=1}^m
\Big(
\frac{\ell(Q_k)}{\ell(Q_k)+|y-c_k|}
\Big)^{\frac{n+N+1}{mn}}.
\end{multline*}
Inequality \eqref{eq.4AA} follows from the definition of the maximal operator.
\end{proof}

%%%%
\begin{lemma}\label{lm.4A4}
Given $w\in A_q$, $1\le q<\infty$, for  $1\le k\le m$  let  $a_k$ be
an $(N,\infty)$ atom supported in $Q_k$ and let $c_k$ be the center of
$Q_k$.  
Suppose $Q_1$ is the cube such that $\ell(Q_1) =  \min\{\ell(Q_k) : 1\le k\le m\}$.
Then
\begin{equation}
\label{eq.4A5}
\|T(a_1,\ldots,a_m)\chi_{Q_1^{*}}\|_{L^{q}(w)}
\lesssim
w(Q_1)^{\frac{1}{q}}
\prod_{l=1}^m
\inf_{z\in Q_1}
M(\chi_{Q_l^{}})(z)^\frac{n+N+1}{mn}.
\end{equation}
\end{lemma}

\begin{proof}
Since the $A_p$ classes are nested, we may assume without loss of
generality that $q>1$.  
  To prove \eqref{eq.4A5} we consider two cases: $Q_1^{*}\cap Q_k^{*}\ne \emptyset$ for all
  $2\le k\le m$ or this intersection is empty for at least one value
  of $k$. In the first case, since
  $\ell(Q_1) = \min\{\ell(Q_k) : 1\le k\le m\}$ we have
  $Q_1^*\subset 3Q_k^{*}$ for all $1\le k\le m$. This implies
\[
\inf_{z\in Q_1}M(\chi_{Q_k})(z)\gtrsim 1,
\]
for all $1\le k\le m$, and so Lemma \ref{lm:TL2w} yields
\begin{multline}
\|T(a_1,\ldots,a_m)\chi_{Q_1^{*}}\|_{L^q(w)}
\le
\|T(a_1,\ldots,a_m)\|_{L^q(w)}\\
\lesssim
\|a_1\|_{L^q(w)}\|a_2\|_{L^\infty}\cdots \|a_m\|_{L^\infty}
\label{eq.4A6}
\lesssim
w(Q_1)^{\frac{1}{q}}
\prod_{k=1}^m\inf_{z\in Q_1}M(\chi_{Q_k})(z)^\frac{n+N+1}{mn}.
\end{multline}

In the second case, since $Q_1^{*}\cap Q_k^{*}=\emptyset$ for some $k$, the set
\[
\Lambda = \{2\le k\le m : Q_1^{*}\cap Q_k^{*}=\emptyset\}
\]
is non-empty.   Fix any point $y\in \mathbb{R}^n$.  Then arguing as
in the previous proof we have that
\begin{equation}
\label{eq.Tay}
T(a_1,\ldots,a_m)(y)=
 \int_{\mathbb{R}^{mn}}K^1(y,y_1,y_2,\ldots,y_m)
 a_1(y_1)\cdots a_m(y_m)d\vec{y},
\end{equation}
where $K^1(y,y_1,\ldots,y_m)$ is defined by \eqref{eq.4A3}.
For $y_1\in Q_1$ we have that for some $\xi_1\in Q_1$ and for all
$y_l\in Q_l$, $1\le l\le m$, 
\begin{equation}
\label{eq.K1y}
\big|K^1(y,y_1,\ldots,y_m)\big|\le
C\ell(Q_1)^{N+1}\Big(
|y-\xi_1|+\sum_{j=2}^m|y-y_j|
\Big)^{-mn-N-1}.
\end{equation}

For all $k\in \Lambda$, since $Q_1^{*}\cap Q_k^{*}=\emptyset$, $|y-\xi_1|+|y-y_k|\ge |\xi_1-y_k| \gtrsim|c_1-c_k|$.
Therefore, for all $y_1\in Q_1^*$ and $y_k\in Q_k$, $k\in \Lambda$,
\[
\big|K^1(y,y_1,\ldots,y_m)\big|\lesssim
\ell(Q_1)^{N+1}\Big(
\sum_{k\in \Lambda}|c_1-c_k|
+\sum_{j=2}^{m}|y-y_j|
\Big)^{-mn-N-1},
\]
If we combine this inequality with \eqref{eq.Tay}, we get
\[
|T(a_1,\ldots,a_m)(y)|
\lesssim
\dfrac{\ell(Q_1)^{n+N+1}}
{\big(\sum_{k\in \Lambda}|c_1-c_k|\big)^{n+N+1}}
\lesssim
\dfrac{\ell(Q_1)^{n+N+1}}
{\big(\sum_{k\in \Lambda}[\ell(Q_1)+|c_1-c_k|+\ell(Q_k)]\big)^{n+N+1}}.
\]
Since $Q_1^*\subset 3Q_l^{*}$ for all $l\notin \Lambda,$
the last inequality gives us
\begin{equation}
\label{eq.4A7}
\|T(a_1,\ldots,a_m)\|_{L^\infty}
\lesssim
\prod_{k=1}^m\inf_{z\in Q_1}M(\chi_{Q_k})(z)^\frac{n+N+1}{mn};
\end{equation}
since $w\in A_q$ is doubling, this implies that
\begin{equation}
\label{eq.4A8}
\|T(a_1,\ldots,a_m)\chi_{Q_1^{*}}\|_{L^q(w)}
\lesssim
w(Q_1)^{\frac{1}{q}}
\prod_{k=1}^m\inf_{z\in Q_1}M(\chi_{Q_k})(z)^\frac{n+N+1}{mn}.
\end{equation}
This completes the proof.
\end{proof}

\section{Proof of Theorem \ref{thm:1}}
\label{section:proof1}

For $1\le k\le m$, let $w_k\in A_\infty$ and
fix arbitrary functions $f_k\in H^{p_k}(w_k)\cap \mathcal{O}_N(\mathbb{R}^n)$. 
By Theorem~\ref{lm.fin.dec}, we have the finite atomic decompositions
\begin{equation}
\label{eq.fk.dec}
  f_k = \sum_{j_k=1}^{N_0}\lambda_{k,j_k}a_{k,j_k},
\end{equation}
  where $\lambda_{k,j_k}\ge 0$ and $a_{k,j_k}$ are $(N,\infty)$-atoms that satisfy
  $$
  \supp(a_{k,j_k})\subset Q_{k,j_k},\quad |{a_{k,j_k}}|\le\chi_ {Q_{k,j_k}},\quad \int_{Q_{k,j_k}}x^{\alpha}a_{k,j_k}(x)dx=0$$
for all ${|\alpha|}\le N$, and
\begin{equation}\label{eq.Hpwk}
\Big\|\sum_{j_k}\lambda_{j_k}\chi_{Q_{j_k}}\Big\|_{L^{p_k}(w_k)}
\le C\|f_k\|_{H^{p_k}(w_k)}.
\end{equation}

Set $\overline{w} = \prod_{k=1}^{m}w_k^{\frac{p}{p_k}}$.
Again by Theorem~\ref{lm.fin.dec}, it will suffice to prove that
\begin{equation}\label{eq.TLpw}
\|T(f_1,\ldots,f_m)\|_{L^p(\overline{w})} \lesssim
\prod_{k=1}^m
\Big\|
\sum_{j_k}{\lambda_{k,j_k}}
\chi_{Q_{k,j_k}}
\Big\|_{L^{p_k}(w_k)}.
\end{equation}

Since $T$ is $m$-linear, we have that  for a.e. $x\in\mathbb R^n$,
  \begin{equation}
\label{eq.4A9}
  T(f_1,\ldots,f_m)(x) =
  \sum_{j_1}\cdots\sum_{j_m}\lambda_{1,j_1}\ldots \lambda_{m,j_m}
  T(a_{1,j_1},\ldots,a_{m,j_m})(x). 
  \end{equation}
Given a cube $Q,$ let $Q^*=2\sqrt{n} Q$.  
 For each $m$-tuple, $(j_1,\ldots,j_m)$,  define $R_{j_1,\ldots,j_m}$
 to be the smallest cube  among $Q_{1,j_1}^*,\ldots,Q_{m,j_m}^*$.
To estimate $\|T(f_1,\ldots,f_m)\|_{L^p(\overline{w})}$ we will
split $T(f_1,\ldots,f_m)$ into two parts:
  $$
  {|T(f_1,\ldots,f_m)(x)|}\le G_1(x) +G_2(x),
  $$
  where 
$$
G_1(x) = \sum_{j_1}\cdots\sum_{j_m}{\lambda_{1,j_1}}\ldots
  {\lambda_{m,j_m}} {|T(a_{1,j_1},\cdots,a_{m,j_m})|}\chi_{R_{j_1,\ldots,j_m}}(x)$$
  and
  $$G_2(x) = \sum_{j_1}\cdots\sum_{j_m}{\lambda_{1,j_1}}\cdots
  {\lambda_{m,j_m}} {|T(a_{1,j_1},\ldots,a_{m,j_m})|}\chi_{(R_{j_1,\ldots,j_m})^c}(x).$$

We first estimate $\|G_2\|_{L^p(\overline{w})}$.  By \eqref{eq.4AA} we
have that
\begin{align*}
|T(a_{1,j_1}, \dots , a_{m,j_m})(x)|
\chi_{(R_{j_1,\ldots,j_m})^c}(x)
\lesssim&
\prod_{k=1}^m
M(\chi_{Q_{k,j_k}})(x)^{\frac{n+N+1}{mn}};
\end{align*}
thus
\begin{equation*}
G_2 \lesssim
\sum_{j_1}\cdots\sum_{j_m}
\prod_{k=1}^m
 {\lambda_{k,j_k}}
  M(\chi_{Q_{k,j_k}})^{\frac{n+N+1}{mn}}
  =
  \prod_{k=1}^m\Big[\sum_{j_k}
 {\lambda_{k,j_k}}
  M(\chi_{Q_{k,j_k}})^{\frac{n+N+1}{mn}}
  \Big].
\end{equation*}
By condition \eqref{eq.N1}, H\"older's inequality and the weighted
Fefferman-Stein vector-valued inequality (see Remark~\ref{FS-argument}), we get
\begin{equation}\label{eq:G2Lpw}
\|G_2\|_{L^p(\overline{w})} \lesssim
\prod_{k=1}^m
\Big\|
\sum_{j_k}{\lambda_{k,j_k}}
\chi_{Q_{k,j_k}}
\Big\|_{L^{p_k}(w_k)}.
\end{equation}

\medskip

We now estimate the norm of $G_1$. Since $\overline{w}\in A_\infty$ by
Lemma \ref{lemma:Ainfty}, we can choose $q> \max(1,p)$ such that
$\overline{w}\in A_q$. Then by Lemma \ref{lm:TL2w} we have that
\begin{equation*}
\Big(
\frac{1}{\overline{w}(R_{j_1,\ldots,j_m})}
\int\limits_{R_{j_1,\ldots,j_m}}
|T(a_{1,j_1},\cdots,a_{m,j_m})|^q(x)\overline{w}(x)dx
\Big)^{\frac1q}
\lesssim
\prod_{k=1}^m
\inf_{z\in R_{j_1,\ldots,j_m}}
M(\chi_{Q_{k,j_k}})(z)^{\frac{n+N+1}{mn}}.
\end{equation*}
If we combine this inequality, Lemma \ref{lm:AvgLpw}, H\"older's
inequality and the  Fefferman-Stein vector-valued inequality imply
that (again see Remark~\ref{FS-argument}), we get the following estimate:
%%%
\begin{align*}
& \|G_1\|_{L^p(\overline{w})} \\
&\qquad \lesssim
\Big\|\sum_{j_1\ldots,j_m}\prod_{k=1}^m{\lambda_{k,j_k}}
\Big(
\frac{1}{\overline{w}(R_{j_1,\ldots,j_m})}
\int\limits_{R_{j_1,\ldots,j_m}}
|T(a_{1,j_1},\cdots,a_{m,j_m})|^q(x)\overline{w}(x)dx
\Big)^{\frac1q}
\chi_{R_{j_1,\ldots,j_m}}
\Big\|_{L^p(\overline{w})}
\\
& \qquad \lesssim
\Big\|\sum_{j_1\ldots,j_m}
\Big(
\prod_{k=1}^m{\lambda_{k,j_k}}
\Big)\cdot
\Big(
\prod_{k=1}^m
\inf_{z\in R_{j_1,\ldots,j_m}}
M(\chi_{Q_{k,j_k}})(z)^{\frac{n+N+1}{mn}}
\Big)
\chi_{R_{j_1,\ldots,j_m}}
\Big\|_{L^p(\overline{w})}\\
&\qquad \lesssim
\Big\|
\prod_{k=1}^m
\Big(
\sum_{j_k}{\lambda_{k,j_k}}
M(\chi_{Q_{k,j_k}})^{\frac{n+N+1}{mn}}
\Big)
\Big\|_{L^p(\overline{w})} \\
&\qquad \lesssim
\prod_{k=1}^m
\Big\|
\sum_{j_k}{\lambda_{k,j_k}}
M(\chi_{Q_{k,j_k}})^{\frac{n+N+1}{mn}}
\Big\|_{L^{p_k}(w_k)}\\
&\qquad \lesssim
\prod_{k=1}^m
\Big\|
\sum_{j_k}{\lambda_{k,j_k}}
\chi_{Q_{k,j_k}}
\Big\|_{L^{p_k}(w_k)}.
\end{align*}
%%% 
If we combine the estimates for $G_1$ and $G_2$, we get the desired
inequality.

\section{Proof of Theorem \ref{thm:2}}
\label{section:proof2}

The proof of Theorem~\ref{thm:2} is very similar to the proof of
Theorem~\ref{thm:1}.  Instead of estimating the norm of $T$, we will
estimate the norm of $M_\phi\circ T$, where $M_\phi$ is 
 the non-tangential maximal operator
\begin{equation*}
M_\phi f(x) = \sup_{0<t<\infty}\sup_{|y-x|<t}|\phi_t*f(y)|,
\end{equation*}
where $\phi\in C_0^\infty$ and $\supp(\phi)\subset B(0,1)$.   We will use  the
that the Hardy space can be characterized by using the non-tangential
maximal function $M_\phi$ with the norm
\[
\|f\|_{H^p(w)}\approx \|M_\phi f\|_{L^p(w)}.
\]
See~\cite{ST89}; this equivalence is guaranteed by our choice of $N_0$
sufficiently large.  Throughout this section we  fix a choice of $\phi$.

In this section, we fix the smooth approximate identity $\phi$ supported in the unit ball. 
The following lemma was first proved in \cite{GNNS17A}; it is the
essential part in the proof of Theorem \ref{thm:2} and so we repeat
the proof here for the convenience of the reader.  Hereafter, given a
cube $Q$, let $Q^{**}=4n Q$.  

\begin{lemma}
\label{lm.mphi}
For $1\le k\le m$, let $a_k$ be $(N,\infty)$ atoms with
$\supp(a_k)\subset Q_k$.
 Suppose that $Q_1$ is such that $\ell(Q_1) =  \min\{\ell(Q_k) : 1\le k\le m\}$.
Then for all $x\notin Q_1^{**}$, we have
\begin{equation}
\label{eq.mphi}
M_\phi T(a_1,\ldots,a_m)(x)\lesssim
 \prod_{l=1}^m 
M(\chi_{Q_l})(x)^{\frac{n+N+1}{m n}}
+
M(\chi_{Q_1})(x)^{\frac{n+s_{\overline{w}}+1}{n}}
\prod_{l=1}^m 
\inf_{z \in Q_1}
M(\chi_{Q_l})(z)^{\frac{N-s_{\overline{w}}}{m n}},
\end{equation}
where $T$ is the operator in Theorem \ref{thm:1}.
\end{lemma}

\begin{proof}
Fix $x\in (Q_1^{**})^c$, $0<t<\infty$ and $y\in \mathbb{R}^n$ such
that $|y-x|<t$. 
To prove \eqref{eq.mphi} it will suffice to show that
\begin{equation}
\label{eq.phiTy}
|\phi_t*T(a_1,\ldots,a_m)(y)| \lesssim \prod_{l=1}^m 
M(\chi_{Q_l})(x)^{\frac{n+N+1}{m n}}
+
M(\chi_{Q_1})(x)^{\frac{n+s_{\overline{w}}+1}{n}}
\prod_{l=1}^m 
\inf_{z \in Q_1}
M(\chi_{Q_l})(z)^{\frac{N-s_{\overline{w}}}{m n}},
\end{equation}
where the implicit constant does not depend on $x,y$ and $t$. We will consider two cases.

\medskip

\subsection*{Case 1:} $t>\frac1{1000n^2}|x-c_1|$. We will exploit the
cancellation in \eqref{eq.can} to show that
\begin{equation}
\label{eq.phiTy1}
|\phi_t*T(a_1,\ldots,a_m)(y)| \lesssim 
M(\chi_{Q_1})(x)^{\frac{n+s_{\overline{w}}+1}{n}}
\prod_{l=1}^m 
\inf_{z \in Q_1}
M(\chi_{Q_l})(z)^{\frac{N-s_{\overline{w}}}{m n}}.
\end{equation}
By  \eqref{eq.can} we have
\begin{multline*}
\phi_t*T(a_1,\ldots,a_m)(y) 
= \int \phi_t(y-z)T(a_1,\ldots,a_m)(z)\, dz\\
= \int \Big(\phi_t(y-z) -
\sum_{|\alpha|\le s_{\overline{w}}}
\frac{\partial^{\alpha}[\phi_t](y-c_1)}{\alpha!}(c_1-z)^{\alpha}\Big)T(a_1,\ldots,a_m)(z)dz.
\end{multline*}
Note that by Taylor's theorem,
\[
\Big|
\phi_t(y-z) -
\sum_{|\alpha|\le s_{\overline{w}}}
\frac{\partial^{\alpha}[\phi_t](y-c_1)}{\alpha!}(c_1-z)^{\alpha}
\Big|
\lesssim \frac{|z-c_1|^{s_{\overline{w}}+1}}{t^{n+s_{\overline{w}}+1}}
\]
for all $y,z\in \mathbb{R}^n$ and all $t\in (0,\infty)$. Since
$t\gtrsim |x-c_1|$ and $x\notin Q_1^{**}$, we have
\begin{align*}
|\phi_t*T(a_1,\ldots,a_m)(y)|\lesssim & 
\int \frac{|z-c_1|^{s_{\overline{w}}+1}}{t^{n+s_{\overline{w}}+1}}|T(a_1,\ldots,a_m)(z)|dz\\
\lesssim&
\Big(\frac{\ell(Q_1)}{|x-c_1|}\Big)^{n+s_{\overline{w}}+1} \frac1{\ell(Q_1)^{n+s_{\overline{w}}+1}}
\int |z-c_1|^{s_{\overline{w}}+1}|T(a_1,\ldots,a_m)(z)|dz\\
\lesssim&
M(\chi_{Q_1})(x)^{\frac{n+s_{\overline{w}}+1}{n}}
\frac1{\ell(Q_1)^{n+s_{\overline{w}}+1}}
\int |z-c_1|^{s_{\overline{w}}+1}|T(a_1,\ldots,a_m)(z)|dz.
\end{align*}
Hence, to prove \eqref{eq.phiTy1} it remains to show that
\begin{equation}
\label{eq.L1T}
\frac1{\ell(Q_1)^{n+s_{\overline{w}}+1}}
\int |z-c_1|^{s_{\overline{w}}+1}|T(a_1,\ldots,a_m)(z)|dz
\lesssim
\prod_{l=1}^m 
\inf_{z \in Q_1}
M(\chi_{Q_l})(z)^{\frac{N-s_{\overline{w}}}{m n}}.
\end{equation}
If we split the integral on the left-hand side of \eqref{eq.L1T}
over $Q_1^*$ and $(Q_1^*)^c$, we can estimate as follows:
\begin{align*}
& \int |z-c_1|^{s_{\overline{w}}+1}|T(a_1,\ldots,a_m)(z)|dz \\
& \qquad \lesssim
\int_{Q_1^*} |z-c_1|^{s_{\overline{w}}+1}|T(a_1,\ldots,a_m)(z)|dz
+\int_{(Q_1^*)^c} |z-c_1|^{s_{\overline{w}}+1}|T(a_1,\ldots,a_m)(z)|dz\\
& \qquad \lesssim
\ell(Q_1)^{s_{\overline{w}}+1}
\int_{Q_1^*} |T(a_1,\ldots,a_m)(z)|dz
+\int_{(Q_1^*)^c} |z-c_1|^{s_{\overline{w}}+1}|T(a_1,\ldots,a_m)(z)|dz.
\end{align*}
By \eqref{eq.4A5}, we can estimate the first integral in the last
inequality by
\begin{equation}
\label{eq.TQ1}
\ell(Q_1)^{s_{\overline{w}}+1}
\int_{Q_1^*} |T(a_1,\ldots,a_m)(z)|dz
\lesssim \ell(Q_1)^{n+s_{\overline{w}}+1}
\prod_{l=1}^m
\inf_{z\in Q_1}
M(\chi_{Q_l^{}})(z)^\frac{n+N+1}{mn}.
\end{equation}
To estimate second integral, we need to exploit carefully the
smoothness of the kernel. Recall the representation of
$T(a_1,\ldots,a_m)(z)$ in \eqref{eq.Tay}.  Denote
$$J=\{2\le l\le m : Q_1^{**}\cap Q_l^{**}=\emptyset\}.$$
For $z\notin Q_1^*$, $\xi_1\in Q_1$, we have
$|z-\xi_1|\approx |z-c_1|\geq \ell(Q_1)$. Also for $l\in J$ and $z_l\in Q_l^*$,
$$|z-\xi_1|+|z-z_l|\geq |\xi_1-z_l|\gtrsim |c_1-c_l|.$$ 
We now estimate $K^1(z,z_1,\ldots,z_m)$ in \eqref{eq.K1y} to get 
\begin{equation*}
|T(a_1,\ldots,a_m)(z)|
\lesssim
\int_{({\mathbb R}^n)^m}
\frac{\ell(Q_1)^{N+1}\chi_{Q_1}(z_1)\,dz_1\cdots dz_m}{
\left(
\ell(Q_1)+|z-c_1|+
\displaystyle
\sum_{l\in J}
|c_1-c_l|
+\sum_{l=2}^m|z-z_l|
\right)^{m n+N+1}}
\end{equation*}
for all $z \in (Q_1^*)^c$.
Thus, 
\begin{align*}
&
\int_{(Q_1^*)^c}|y-c_1|^{s_{\overline{w}}+1}|T(a_1,\ldots,a_m)(y)|\,dy
\\
&\qquad \lesssim
\int_{{\mathbb R}^n \times ({\mathbb R}^n)^m}
\frac{|y-c_1|^{s_{\overline{w}}+1}\ell(Q_1)^{N+1}\chi_{Q_1}(y_1)\,d\vec{y}dy}
{\displaystyle\bigg(\ell(Q_1)+|y-c_1|
+\sum_{l\in J}|c_1-c_l|
+\sum_{l=2}^m |y-y_l|\bigg)^{m n+N+1}}
\\
&\qquad \lesssim 
\ell(Q_1)^{n+s_{\overline{w}}+1}
\prod_{l\in J}
\Big(
\frac{\ell(Q_l)}{\ell(Q_1)+|c_1-c_l|}
\Big)^\frac{N-s_{\overline{w}}}{m}.
\end{align*}
Note that
$1\lesssim\inf_{z\in Q_1}M(\chi_{Q_l})(z)$
if $Q_1^{**} \cap Q_l^{**} \ne \emptyset$ and for all $l\in J$,
\[
\frac{\ell(Q_l)}{\ell(Q_1)+|c_1-c_l|}
\lesssim \inf_{z\in Q_1}M(\chi_{Q_l})(z)^\frac{1}{n}.
\]
Therefore,
\begin{equation}
\label{eq.TQ2}
\int_{(Q_1^*)^c}|y-c_1|^{s_{\overline{w}}+1}|T(a_1,\ldots,a_m)(y)|\,dy
\lesssim
\ell(Q_1)^{n+s_{\overline{w}}+1}
\prod_{l=1}^m
\inf_{z \in Q_1} M(\chi_{Q_l})(z)^{\frac{N-s_{\overline{w}}}{mn}}.
\end{equation}
Now we combine \eqref{eq.TQ1} and \eqref{eq.TQ2} we get
\eqref{eq.L1T},  which completes the proof of Case 1.

\medskip

\subsection*{Case 2:} $t\le \frac1{1000n^2}|x-c_1|$. In this case, we
will show that
\begin{equation}
\label{eq.phiTy2}
|\phi_t*T(a_1,\ldots,a_m)(y)| \lesssim 
\prod_{l=1}^m 
M(\chi_{Q_l})(x)^{\frac{n+N+1}{m n}}.
\end{equation}
Since $\supp(\phi)\subset B(0,1)$ and $|y-x|<t$,
\begin{multline} \label{eq.supTBx}
|\phi_t*T(a_1,\ldots,a_m)(y)| 
\le
\int_{B(y,t)} t^{-n}\big|
\phi\big(t^{-1}(y-z)\big)
T(a_1,\ldots,a_m)(z)
\big|\, dz\\
\lesssim
\sup_{z\in B(y,t)}|T(a_1,\ldots,a_m)(z)|
\lesssim
\sup_{z\in B(x,2t)}|T(a_1,\ldots,a_m)(z)|.
\end{multline}

Let $\Lambda = \{1\le l\le m : x\notin Q_k^{**}\}$.
For $z\in B(x,2t)$, $\xi_1\in Q_1$, we have
\[
|x-c_1|\le |x-z| + |z-c_1|\le 2t + |\xi_1-c_1|+|z-\xi_1| \le
\frac1{500n^2}|x-c_1| +\frac12|x-c_1| + |z-\xi_1|;
\]
hence,
\[
t\lesssim |x-c_1|\lesssim |z-\xi_1|.
\]
For $l\in \Lambda$ and $z_l\in Q_l$, since $x\notin Q_l^{**}$,
\[
|x-c_l|\le 2|x-z_l|\le 2|x-z|+2|z-z_l|\le 4t+ 2|z-z_k|\lesssim |z-\xi_1|+ |z-z_k|.
\]
Recall the formula for $T(a_1,\ldots,a_m)(z)$ in \eqref{eq.Tay}; we 
estimate $K^1(z,z_1,\ldots,z_m)$ in \eqref{eq.K1y} to get
\begin{equation*}
|T(a_1,\ldots,a_m)(z)|
\lesssim
\int_{(\mathbb{R}^n)^m}
\frac{\ell(Q_1)^{N+1}\chi_{Q_1}(z_1)\,dz_1\cdots dz_m}
{\Big(\sum_{l=2}^m |z-z_l|+
\sum_{k \in \Lambda} |x-c_k|\Big)^{m n+N+1}}
\end{equation*}
for all $z\in B(x,2t)$.  From this we get that
\begin{equation}
\label{eq.MQl}
\sup_{z\in B(x,2t)}|T(a_1,\ldots,a_m)(z)|
\lesssim
\prod_{l\in \Lambda}
\frac{\ell(Q_l)^{\frac{n+N+1}{|\Lambda|}}\chi_{(Q_l^{**})^c}(x)}
{|x-c_l|^{\frac{n+N+1}{|\Lambda|}}}
\cdot
\prod_{k\notin \Lambda}\chi_{Q_k^{**}}(x)
\lesssim
\prod_{l=1}^m M(\chi_{Q_l})(x)^{\frac{n+N+1}{m n}}.
\end{equation}
Combining \eqref{eq.supTBx} and \eqref{eq.MQl} gives
\eqref{eq.phiTy2}. This completes Case 2 and so completes the proof.
\end{proof}

The next lemma is an immediate consequence  of Lemma \ref{lm.4A4} and
the fact that $M_\phi$ is bounded on $L^q(w)$ if $w\in A_q$ (since it
is controlled pointwise by the Hardy-Littlewood maximal operator;
cf.~\cite{garcia-cuerva-rubiodefrancia85}). 

\begin{lemma}
\label{lm.5A4}
Given $w\in A_q$, $1\le q<\infty$, for  $1\le k\le m$  let  $a_k$ be an $(N,\infty)$ atom
supported in $Q_k$.  Suppose $Q_1$ is the cube
such that $\ell(Q_1) = \min\{\ell(Q_k) : 1\le k\le m\}$.  Then
\begin{align}
\label{eq.5A5}
\|M_\phi T(a_1,\ldots,a_m)\chi_{Q_1^{**}}\|_{L^{q}(w)}&\lesssim
w(Q_1)^{\frac{1}{q}}
\prod_{l=1}^m
\inf_{z\in Q_1}
M(\chi_{Q_l^{}})(z)^\frac{n+N+1}{mn}.
\end{align}
\end{lemma}

\medskip

\begin{proof}[Proof of Theorem \ref{thm:2}]
  Fix $w_k\in A_\infty$, $1\le k\le m$, and define
  $\overline{w} = \prod_{k=1}^mw_k^{\frac{p}{p_k}}$. Fix 
  $f_k\in H^{p_k}(w_k)\cap \mathcal{O}_N(\mathbb{R}^n)$,
  $1\le k\le m$. We will show that
\begin{equation}
\label{eq.MphiT1}
\|M_\phi T(f_1,\ldots,f_m)\|_{L^p(\overline{w})}
\lesssim
\|f_1\|_{H^{p_1}(w_1)}\cdots \|f_m\|_{H^{p_m}(w_m)}.
\end{equation}
Form the atomic decompositions of the functions $f_k$ as in the proof
of Theorem \ref{thm:1} to get \eqref{eq.fk.dec} and
\eqref{eq.Hpwk}. Then to prove \eqref{eq.MphiT1}, it is enough show
that
\begin{equation}
\label{eq.MphiT2}
\|M_\phi T(f_1,\ldots,f_m)\|_{L^p(\overline{w})}
\lesssim
\prod_{k=1}^m
\Big\|
\sum_{j_k}{\lambda_{k,j_k}}
\chi_{Q_{k,j_k}}
\Big\|_{L^{p_k}(w_k)}.
\end{equation}
Since  $M_\phi\circ T$ is multi-sublinear, we can write
\[
M_\phi T(f_1,\ldots,f_m)(x)\le G_1(x) +G_2(x),
\]
where 
$$G_1(x) = \sum_{j_1}\cdots\sum_{j_m}{\lambda_{1,j_1}}\ldots
  {\lambda_{m,j_m}} {M_\phi T(a_{1,j_1},\cdots,a_{m,j_m})}\chi_{R_{j_1,\ldots,j_m}}(x)$$
  and
  $$G_2(x) = \sum_{j_1}\cdots\sum_{j_m}{\lambda_{1,j_1}}\cdots
  {\lambda_{m,j_m}} {M_\phi
    T(a_{1,j_1},\ldots,a_{m,j_m})}\chi_{(R_{j_1,\ldots,j_m})^c}(x).$$
Here $R_{j_1,\ldots,j_m}$ is the smallest cube among $Q_{1,j_1}^{**},\ldots,Q_{m,j_m}^{**}$.

A similar argument as in the proof of Theorem \ref{thm:1} with Lemma \ref{lm.5A4} in place of Lemma \ref{lm.4A4} gives
\begin{equation}
\label{eq:G1Mphi}
\|G_1\|_{L^p(\overline{w})} 
\lesssim
\prod_{k=1}^m
\Big\|
\sum_{j_k}{\lambda_{k,j_k}}
\chi_{Q_{k,j_k}}
\Big\|_{L^{p_k}(w_k)}.
\end{equation}

We now estimate the norm of $G_2$.  By Lemma \ref{lm.mphi} we get that
\[
G_2(x)\lesssim G_{21}(x)+ G_{22}(x),
\]
where
\[
G_{21}(x) =  \sum_{j_1}\cdots\sum_{j_m}{\lambda_{1,j_1}}\cdots
  {\lambda_{m,j_m}} \prod_{k=1}^mM(\chi_{Q_{k,j_k}})(x)^{\frac{n+N+1}{m n}}
\]
and
\[
G_{22}(x) =  \sum_{j_1}\cdots\sum_{j_m}{\lambda_{1,j_1}}\cdots
  {\lambda_{m,j_m}} M(\chi_{R_{j_1,\ldots,j_m}})(x)^{\frac{n+s_{\overline{w}}+1}{n}}
  \prod_{l=1}^m 
\inf_{z \in R_{j_1,\ldots,j_m}}
M(\chi_{Q_{l,j_l}})(z)^{\frac{N-s_{\overline{w}}}{m n}}.
\]
The function $G_{21}$ can be estimated by essentially the same
argument used for $G_1$ to get
\begin{equation}
\label{eq:G21Mphi}
\|G_{21}\|_{L^p(\overline{w})} 
\lesssim
\prod_{k=1}^m
\Big\|
\sum_{j_k}{\lambda_{k,j_k}}
\chi_{Q_{k,j_k}}
\Big\|_{L^{p_k}(w_k)}.
\end{equation}

To estimate  $G_{22}$, since $\frac{(n+s_{\overline{w}}+1)p}{n}>1$,
we use \eqref{eq.N2} and the Fefferman-Stein vector-valued inequality
(cf. Remark~\ref{FS-argument}) to get
\begin{align}
\|G_{22}\|_{L^p(\overline{w})} 
\lesssim&
\Big\|
\sum_{j_1}\cdots\sum_{j_m}{\lambda_{1,j_1}}\cdots
  {\lambda_{m,j_m}} \chi_{R_{j_1,\ldots,j_m}} \prod_{l=1}^m 
\inf_{z \in R_{j_1,\ldots,j_m}}
M(\chi_{Q_{l,j_l}})(z)^{\frac{N-s_{\overline{w}}}{m n}}
\Big\|_{L^p(\overline{w})} 
\notag
\\
\le&
\Big\|
\sum_{j_1}\cdots\sum_{j_m}{\lambda_{1,j_1}}\cdots
  {\lambda_{m,j_m}}  \prod_{k=1}^m 
M(\chi_{Q_{k,j_k}})^{\frac{N-s_{\overline{w}}}{m n}}
\Big\|_{L^p(\overline{w})}
\notag
\\
\le&
\prod_{k=1}^m
\Big\|
\sum_{j_k}{\lambda_{k,j_k}}
M(\chi_{Q_{k,j_k}})^{\frac{N-s_{\overline{w}}}{mn}}
\Big\|_{L^{p_k}(w_k)}
\label{eq:G22Mphi}
\lesssim
\prod_{k=1}^m
\Big\|
\sum_{j_k}{\lambda_{k,j_k}}
\chi_{Q_{k,j_k}}
\Big\|_{L^{p_k}(w_k)}.
\end{align}

If  we combine \eqref{eq:G1Mphi}, \eqref{eq:G21Mphi} and
\eqref{eq:G22Mphi}, we get \eqref{eq.MphiT2} and this completes the proof.
\end{proof}

\section{Variable Hardy spaces:  proof of Theorems~\ref{thm:3} and~\ref{thm:4}}
\label{section:variable}

In this section we prove Theorems~\ref{thm:3}
and~\ref{thm:4}.  In fact, we will prove two more general results that
include these theorems as special cases.  To do so, we first recall
some basic facts about the variable Lebsesgue spaces.  For complete
information we refer the reader to~\cite{CF13}.

Let $\Pp_0(\R^n)$ be the set of all measurable functions $\pp : \R^n
\rightarrow (0,\infty)$.  Define
\[ p_- = \essinf_{x\in \mathbb{R}^n} p(x), 
\qquad p_+ = \esssup_{x\in \mathbb{R}^n}
p(x). \]
Given $\pp\in \Pp_0(\R^n)$ define $\Lp=\Lp(\R^n)$ to be the set of all
measurable functions $f$ such that for some $\lambda>0$,
\[ \rho(f/\lambda) = \int_{\mathbb{R}^n}
\left(\frac{|f(x)|}{\lambda}\right)^{p(x)}\,dx< \infty.  \]
This becomes a quasi-Banach space with the ``norm''
$$ \|f\|_{L^{p(\cdot)}} = \inf\left\{ \lambda > 0 : \rho(f/\lambda)\leq 1
\right\}.
$$
If $p_-\geq 1$, $\Lp$ is a Banach space; if $\pp=p$ a constant, then
$\Lp=L^p$ with equality of norms.

If the maximal operator is bounded on $\Lp$ we write that $\pp \in
\B$.  A necessary condition for this to be the case is that $p_->1$.
A sufficient condition is that $1<p_-\leq p_+<\infty$ and $\pp$ is
log-H\"older continuous: i.e.,~\eqref{eqn:log-holder1}
and~\eqref{eqn:log-holder2} hold.  However, this continuity condition
is not necessary:  see~\cite{CF13} for a detailed discussion of this
problem. 

Given $\pp\in \Pp_0(\R^n)$, the variable Hardy space $H^\pp$ is
defined to be the set of all distributions $f$ such that $M_{N_0}f\in
\Lp$.  Again, we here assume $N_0>0$ is a sufficiently large constant
so that all the standard definitions of the classical Hardy spaces are
equivalent.  These spaces were examined in detail in~\cite{CW14} (see
also~\cite{NS12}). 

A very important tool for proving norm inequalities in spaces of
variable exponents is the extension of the Rubio de Francia theory of
extrapolation to the scale of variable Lebesgue spaces.  For the
history and application of this approach for linear operators,
see~\cite{CF13,CMP11}.   To prove Theorems~\ref{thm:3} and~\ref{thm:4}
we will use a multilinear version due to the first author and
Naibo~\cite{CN16}.  They only stated their proof for the bilinear
case, but the same proof immediately extends to the general multilinear setting.

\begin{theorem}\label{thm:xtpl}
  Let $\F = \big\{ (f_1,\ldots,f_m,F)\big\}$ be a family of
  $(m+1)$-tuples of non-negative, measurable functions on $\R^n$.
  Suppose that there exist indices $0<p_1,\ldots,p_m,p<\infty$
  satisfying $\frac{1}{p} = \frac{1}{p_1}+\cdots+\frac{1}{p_m}$ such
  that for all weights $w_k\in A_1$, $1\leq k\leq m$, and
  $\overline{w} = \prod_{k=1}^{m}w_k^{\frac{p}{p_k}}$, 
\begin{equation} \label{eqn:xtpl1}
\|F\|_{L^{p}(\overline{w})} \lesssim \|f_1\|_{L^{p_1}(w_1)}\cdots \|f_m\|_{L^{p_m}(w_m)}
\end{equation}
for all $(f_1,\ldots,f_m,F)$ such that $F\in L^{p}(\overline{w})$,
and where the implicit constant depends only on $n$, $p_k$ and
$[w_k]_{A_1}$, $1\leq k \leq m$.   Let
$q_1(\cdot),\ldots,q_m(\cdot),q(\cdot)\in \Pp_0$ be  such that 
\[
\frac{1}{q(\cdot)} = \frac{1}{q_1(\cdot)}+\cdots+\frac{1}{q_m(\cdot)},
\]
$p_k< (q_k)_-$,  $1\le k\le m$, and $q_k(\cdot)/p_k\in \B$.  Then
\begin{equation} \label{eqn:xtpl2}
\|F\|_{L^{q(\cdot)}} \lesssim \|f_1\|_{L^{q_1(\cdot)}}\cdots \|f_m\|_{L^{q_m(\cdot)}}
\end{equation}
provided $\|F\|_{L^{q(\cdot)}}<\infty$.  The implicit constant only depends
on $n$ and $q_k(\cdot)$, $1\leq k \leq m$. 
\end{theorem}

\begin{remark}
In~\cite{CN16}, the hypothesis on the exponents $q_k(\cdot)$ was
stated as $(q_k(\cdot)/p_k)' \in \B$, where this exponent is the
conjugate exponent, defined pointwise by
$\frac{1}{p(x)}+\frac{1}{p'(x)}=1$.  It was stated in this way for
technical reasons related to the proof.  However, these two hypotheses
are equivalent:  see~\cite[Corollary~4.64]{CF13}.
\end{remark}

The one technical obstacle in applying Theorem~\ref{thm:xtpl} is
constructing the family $\F$ to satisfy the hypotheses that the
left-hand sides of \eqref{eqn:xtpl1} and \eqref{eqn:xtpl2} are
finite and that the resulting family is large enough that the desired
result can be proved via a density argument.  In our case we will use
the atomic decomposition in the weighted and variable Hardy spaces.
As we noted in Section~\ref{section:weights}, given $w\in A_\infty$
and $0<p<\infty$, every $f\in H^p(w)$ can be written as the sum
\begin{equation} \label{eqn:atm-dcp}
 f = \sum_k \lambda_k a_k, 
\end{equation}
where $\lambda_k\geq 0$ and the $a_k$ are $(N,\infty)$ atoms, provided
$N\geq s_w$.  Moreover, this series converges both in the sense of
distributions and in $H^p(w)$. 
(See~\cite[Chapter VIII]{ST89}.)  The same is true in the variable
Hardy spaces.   More precisely:  suppose $\pp \in \Pp_0$ is such that there
exists $0<p_0<p_-$ with $\pp/p_0 \in \B$.  Then given $N>
n(p_0^{-1}-1)$, if $f\in H^\pp$, there exists a sequence of
$(N,\infty)$ atoms $a_k$ and constants $\lambda_k$ such that
\eqref{eqn:atm-dcp} holds, and the series converges both in the sense
of distributions and in $H^\pp$.  (See~\cite[Theorem~6.3]{CW14}; here
we have slightly modified the definition of atoms, but the change is
immediate.)  It follows immediately from these two results that finite
sums of $(N,\infty)$ atoms, for $N$ sufficiently large, are dense in
$H^p(w)$ and $H^\pp$.  

\begin{remark}
In applying the density of finite sums of atoms, we are not making use
of the finite atomic decomposition norm (as in Theorem~\ref{lm.fin.dec} for weighted
spaces or in the corresponding result for variable Hardy spaces
in~\cite{CW14}).  We will only use that these sums are dense with
respect to the given Hardy space norm.
\end{remark}

\begin{theorem} \label{thm:Hp-Lp}
Let $q_1(\cdot),\ldots,q_m(\cdot),\qq \in \Pp_0$ be such that
$\frac{1}{\qq} = \frac{1}{q_1(\cdot)}+\cdots+\frac{1}{q_m(\cdot)}$ and
$0<(q_k)_- \leq (q_k)_+< \infty$, $1\leq k \leq m$.  Suppose further
that there exist $0<p_1,\ldots,p_m,p<\infty$ with
$\frac{1}{p}=\frac{1}{p_1}+\cdots+\frac{1}{p_m}$, $0<p_k<(q_k)_-$, and
$q_k(\cdot)/p_k \in \B$.  If $T$ is an $m$-CZO as in
Theorem~\ref{thm:1} satisfying~\eqref{eqn:kernelregularity} for all
$|\alpha|\leq N$, where
\[ N\ge \max\bigg\{\floor{mn\Big(\frac{1}{p_k}-1\Big)}_+,1\le
  k\le m\bigg\}+(m-1)n, \]
then
\[ T : H^{q_1(\cdot)}\times \cdots \times H^{q_1(\cdot)} \rightarrow
  L^\qq. \]
\end{theorem}

\begin{remark}
Theorem~\ref{thm:3} follows at once since if $q_k(\cdot)$ is
log-H\"older continuous, then $q_k(\cdot)/p_k \in \B$.
\end{remark}

\begin{proof}
Fix an integer $K_0$ such that
\[ K_0 > \max \bigg\{ \floor{n\Big(\frac{1}{p_k}-1\Big)}_+, 1 \leq k
  \leq m \bigg\}. \]
Define the family $\F = \big\{ (f_1,\ldots,f_m,F) \big\}$, where for
each $1\leq k\leq m$, 
\[ f_k = \sum_{j=1}^L \lambda_j a_j \]
is a finite linear combination of $(K_0,\infty)$ atoms, and 
\[ F = \max\big\{ |T(f_1,\ldots,f_m)|, R\big\} \chi_{B(0,R)}, \]
where $0<R<\infty$.  

Now fix any collection of weights $w_1,\ldots,w_m \in A_1$.  Then for
$1\leq k\leq m$, $r_{w_k}=1$, so $K_0>s_{w_k}$.  Therefore, given any
$(m+1)$-tuple $(f_1,\ldots,f_m,F)\in \F$, $f_k \in H^{p_k}(w_k)$, and
by Theorem~\ref{thm:1},
\[ \|F\|_{L^p(\overline{w})} \leq
\|T(f_1,\ldots,f_m)\|_{L^p(\overline{w})} 
\lesssim \|f_1\|_{H^{p_1}(w_1)}\cdots \|f_m\|_{H^{p_m}(w_m)} <
\infty. \]

Moreover, we have that $f_k \in H^{q_k(\cdot)}$ and 
\[ \|F\|_\qq \leq R\|\chi_{B(0,R)}\|_\qq < \infty. \]

Hence, by Theorem~\ref{thm:xtpl} we have that 
\[ \|F\|_{\qq} \lesssim \|f_1\|_{H^{q_1(\cdot)}}\cdots
  \|f_m\|_{H^{q_m(\cdot)}} < \infty.  \]
By Fatou's lemma in the scale of variable Lebesgue
spaces~\cite[Theorem~2.61]{CF13}, we get
\[ \|T(f_1,\ldots,f_m)\|_{\qq} \lesssim \|f_1\|_{H^{q_1(\cdot)}}\cdots
  \|f_m\|_{H^{q_m(\cdot)}} < \infty.  \]
Since finite sums of $(K_0,\infty)$ atoms are dense in
$H^{q_k(\cdot)}$, $1\leq k \leq m$, a standard density argument shows
that this inequality holds for all $f_k \in H^{q_k(\cdot)}$, $1\leq k
\leq m$.  This completes the proof.
\end{proof}

\medskip

The proof of the following result is identical to the proof of
Theorem~\ref{thm:Hp-Lp}, except that in the definition of the family
$\F$ we replace $T$ by $M_{N_0}T$ (for $N_0$ sufficiently large) and
use Theorem~\ref{thm:2} instead of Theorem~\ref{thm:1}.
Theorem~\ref{thm:4} again follows as an immediate corollary.

\begin{theorem} \label{thm:Hp-Hp}
Given $q_1(\cdot),\ldots,q_m(\cdot),\qq$ and $p_1,\ldots,p_m,p$ as in
Theorem~\ref{thm:Hp-Lp}, let $T$ be an $m$-CZO as in
Theorem~\ref{thm:1} satisfying~\eqref{eqn:kernelregularity} for all
$|\alpha|\leq N$, where
\[ N\ge \floor{mn\Big(\frac{1}{p}-1\Big)}_+ + \max\bigg\{\floor{mn\Big(\frac{1}{p_k}-1\Big)}_+,1\le
  k\le m\bigg\}+mn. \]
Suppose further that $T$ satisfies~\eqref{eq.can} for all $|\alpha|
\leq \floor{n(1/p-1)}_+$.  Then
\[ T : H^{q_1(\cdot)}\times \cdots \times H^{q_1(\cdot)} \rightarrow
  H^\qq. \]
\end{theorem}

\bibliographystyle{amsplain}

\end{document}